\documentclass[12pt,letterpaper]{article}

\usepackage{renstyles}

\title{Recovering coefficients in a system of semilinear Helmholtz equations from internal data}

\author{
Kui Ren\thanks{
		Department of Applied Physics and Applied Mathematics, Columbia University, New York, NY 10027; \href{mailto:kr2002@columbia.edu}{kr2002@columbia.edu}
	}
\and
Nathan Soedjak\thanks{
		Department of Applied Physics and Applied Mathematics, Columbia University, New York, NY 10027; \href{mailto:ns3572@columbia.edu}{ns3572@columbia.edu}
	}
}

\begin{document}

\maketitle


\begin{abstract}
    We study an inverse problem for a coupled system of semilinear Helmholtz equations where we are interested in reconstructing multiple coefficients in the system from internal data measured in applications such as thermoacoustic imaging. We derive results on the uniqueness and stability of the inverse problem in the case of small boundary data based on the technique of first- and higher-order linearization. Numerical simulations are provided to illustrate the quality of reconstructions that can be expected from noisy data.
\end{abstract}
  
\begin{keywords}
	Inverse problems, semilinear Helmholtz equation, uniqueness, stability, thermoacoustic imaging, nonlinear physics, second harmonic generation
\end{keywords}

\begin{AMS}
	35R30, 49M41, 65N21, 78A46
\end{AMS}


\section{Introduction}

Let $\Omega\subset\bbR^d$ ($d\ge 2$) be a bounded domain with smooth boundary $\partial\Omega$. We consider the following system of coupled semilinear Helmholtz equations 
\begin{equation}\label{EQ:SHG}
	\begin{array}{rcll}
  	\Delta u+ k^2 (1+\eta) u +ik\sigma u&=& -k^2 \gamma u^* v, & \mbox{in}\ \ \Omega\\
  	\Delta v+ (2k)^2 (1+\eta) v +i2k\sigma v &=& -(2k)^2 \gamma u^2,  &\mbox{in}\ \ \Omega\\
	u = g, && v = h, & \mbox{on}\ \partial\Omega
	\end{array}
\end{equation}
where $\Delta$ denotes the standard Laplacian operator, and $u^*$ denotes the complex conjugate of $u$. This system serves as a simplified model of the second harmonic generation process in a heterogeneous medium excited by an incident wave source $g$~\cite{BoWo-Book99,Boyd-Book20,BoLiMaSc-IP17,YuYa-JOSA13,SzKi-JOSA18,FrTeGaBlNeMa-JOSA15,ZeHoLiKoMo-PRB09}. The fields $u$ and $v$ are, respectively, the incident field (with wave number $k$) and the generated second-harmonics (with wave number $2k$). The medium has first- and second-order susceptibility $\eta$ and $\gamma$, respectively, and an absorption coefficient $\sigma$. 

We are interested in inverse problems to system~\eqref{EQ:SHG} where the objective is to reconstruct the coefficients in the system from data of the form:
\begin{equation}\label{EQ:Data Form}
	H(\bx) = \Gamma(\bx) \sigma(\bx) (|u|^2+|v|^2),\qquad \bx\in\Omega\,.
\end{equation}
where $\Gamma(\bx)$ is an additional physical coefficient that appears in the data generation process. This inverse problem is motivated by applications in thermoacoustic imaging, a hybrid imaging modality where thermoacoustic effect is used to couple high-resolution ultrasound imaging to microwave imaging to achieve high-resolution and high-contrast imaging of physical properties of heterogeneous media in the microwave regime. In thermoacoustic imaging, $H$ is the initial pressure field of the ultrasound generated by the thermoacoustic effect. It is proportional to the local energy absorbed by the medium from microwave illumination, that is, $\sigma(|u|^2+|v|^2)$. The proportional constant $\Gamma(\bx)$ is called the Gr\"uneisen coefficient~\cite{BaRe-IP12}. We refer interested readers to~\cite{BaZh-IP14,BeBrPr-IP19,CrLiSh-M2AS23,JeEl-IP20} and references therein for the recent development in the modeling and computational aspects of thermoacoustic imaging.

There are two main differences between the inverse problem we study here and those that exist in the literature. First, our model~\eqref{EQ:SHG} takes into account second-harmonic generation, a nonlinear mechanism that is often used for the imaging of molecular signatures of particular proteins in biomedical applications. Second, the objective of our inverse problem includes the Gr\"uneisen coefficient $\Gamma$, which is mostly ignored in the previous studies of quantitative thermoacoustic imaging~\cite{AkBeDaElLiMi-JIIP17,AmGaJiNg-ARMA12,BaZh-IP14,BaReUhZh-IP11}.

The fact that the absorbed energy is in the form of $\sigma(\bx)(|u|^2+|v|^2)$ has to be understood from the physics of the thermoacoustic process. In a nutshell, consider the full time-dependent forms of the incident (at frequency $\omega$) and generated (at frequency $2\omega$) electric wave of the form:
\begin{align*}
    u(\bx,t) &= 2\Re(u(\bx)e^{-i\omega t}) = 2|u(\bx)|\cos(\varphi_u(\bx) - \omega t)\,,\\
    v(\bx,t) &= 2\Re(v(\bx)e^{-i2\omega t}) = 2|v(\bx)|\cos(\varphi_v(\bx) - 2\omega t)\,,
\end{align*}
where $\varphi_u$ (resp. $\varphi_v$) is the phase of $u$ (resp. $v$). Let $I(\bx)$ denote the energy density of the total electric field at the location $\bx$, averaged over a period of length $T:=2\pi/\omega$. It is then clear that
\begin{align*}
    I(\bx) &= \frac{1}{T}\int_0^T \frac{1}{2} |u(\bx,t)+v(\bx,t)|^2 \,dt = \frac{1}{2T}\int_0^T u(\bx,t)^2 + v(\bx,t)^2 + 2u(\bx,t)v(\bx,t)\, dt\\
    &=4|u(\bx)|^2\frac{1}{2T}\int_0^T\cos^2(\varphi_u(\bx) - \omega t)\,dt + 4|v(\bx)|^2\frac{1}{2T}\int_0^T\cos^2(\varphi_v(\bx) - 2\omega t)\,dt\\
    &\qquad + 8|u(\bx)||v(\bx)|\frac{1}{2T}\int_0^T \cos(\varphi_u(\bx) - \omega t)\cos(\varphi_v(\bx) - 2\omega t)\,dt\\
    &=(|u(\bx)|^2+|v(\bx)|^2)\,,
\end{align*}
where we have used the standard trigonometic identity $\cos(x)\cos(y)=\frac{1}{2}(\cos(x+y)+\cos(x-y))$ to simplify the integrals.
Therefore, the cross-term vanishes due to orthogonality. The absorbed radiation at location $\bx$ is thus $\sigma(\bx)I(\bx)=\sigma(\bx)(|u|^2+|v|^2)$. This simple calculation provides a (maybe overly-simplified) justification of the data~\eqref{EQ:Data Form} as the internal data in thermoacoustic imaging with second-harmonic generation.

The main objective of this paper is to study the problem of determining information on $(\Gamma, \eta, \gamma, \sigma)$ from information encoded in the map:
\begin{equation}\label{EQ:Data}
	\Lambda_{\Gamma, \eta, \gamma, \sigma}: (g,h) \mapsto H.
\end{equation}
We will show that under appropriate conditions, the data~\eqref{EQ:Data} allow unique (and stable, in an appropriate sense) reconstruction of the coefficients $(\Gamma, \eta, \gamma, \sigma)$. Moreover, there is an explicit reconstruction method to recover $(\Gamma, \eta, \sigma)$ (see the proof of Theorem~\ref{THM:Uniqueness}), and another explicit method to reconstruct $\gamma$ (see the remarks below ~\eqref{EQ:Lin Sys}).

The paper is organized as follows. We first review in Section~\ref{SEC:Forward} some of the elementary properties of the model~\eqref{EQ:SHG} that we will use in our analysis. We also introduce the multilinearization method as the basis of the study of the inverse problems. We then derive the uniqueness and stability of reconstructing $(\Gamma, \eta, \sigma)$ in Section~\ref{SEC:IP Step I} and study the problem of reconstructing $\gamma$ in Section~\ref{SEC:IP Step II}. Numerical simulations based on synthetic data will be provided in Section~\ref{SEC:Numerics} to demonstrate the quality of the reconstructions that can be achieved in such an inverse problem before we conclude the paper with additional remarks in Section~\ref{SEC:Concl}.

\section{The forward model and its linearization}
\label{SEC:Forward}

Throughout the paper, we make the following assumptions on the domain $\Omega$ and the physical coefficients involved in the inverse problem:
\begin{enumerate}[label=($\cA$-\roman*)]
\item \label{ASS:Domain} The domain $\Omega$ is bounded with smooth boundary $\partial\Omega$.
\item \label{ASS:Coefficients} The coefficients $\Gamma,\eta,\sigma,\gamma$ all lie in the set 
\begin{align*}
\cY := \{f\in \cC^2(\overline{\Omega}; \bbR): c_1\le f\le c_2\ \mbox{in}\ \overline{\Omega}\}   
\end{align*}
for some $c_1>0$ and $c_2>0$.
\end{enumerate}
While it is clear that such assumptions can be slightly relaxed for the technical results in the rest of the paper to still hold, we choose the current form to make the presentation of the paper easy to follow.

\subsection{Well-posedness of the forward model}
\label{SUBSEC:Well-posedness}

We start with the well-posedness of the semilinear system~\eqref{EQ:SHG} for small boundary data.
\begin{theorem}\label{THM:Well-Posedness} 
Let $\alpha\in (0,1)$. Under the assumptions in~ \ref{ASS:Domain} and~\ref{ASS:Coefficients}, there exist $\eps>0$ and $\delta>0$ such that for all $g,h\in \cC^{2,\alpha}(\partial\Omega;\bbC)$ with $\|g\|_{\cC^{2,\alpha}(\partial\Omega)}<\eps$ and $\|h\|_{\cC^{2,\alpha}(\partial\Omega)}<\eps$, the boundary value problem \eqref{EQ:SHG} has a unique solution 
\begin{align*}
(u,v)\in \{f\in \cC^{2,\alpha}(\overline{\Omega};\bbC): \|f\|_{\cC^{2,\alpha}(\overline{\Omega})}\le \delta\}^2.    
\end{align*}
Moreover, there exists a constant $C=C(\alpha,\Omega,\eta,\sigma,\gamma)$ such that this unique solution satisfies the estimates
\begin{equation}\label{EQ:SHG Estimates}
\begin{aligned}
    \|u\|_{\cC^{2,\alpha}(\overline{\Omega})} &\le C( \|g\|_{\cC^{2,\alpha}(\partial\Omega)} + \|h\|_{\cC^{2,\alpha}(\partial\Omega)}),\\
    \|v\|_{\cC^{2,\alpha}(\overline{\Omega})} &\le C( \|g\|_{\cC^{2,\alpha}(\partial\Omega)} + \|h\|_{\cC^{2,\alpha}(\partial\Omega)}).
\end{aligned}
\end{equation}
\end{theorem}
This result comes as a more-or-less straightforward application of the Banach fixed point theorem in a standard setting. For the convenience of the readers, we provide the proof in Appendix~\ref{APP:Well-posedness}.

The above well-posedness result is not satisfactory as it requires that the boundary data to be small. Currently, we do not have a stronger result. This result, however, is sufficient for the inverse problem we want to study as our method in the next sections will be mainly based on the linearization of the forward model with small boundary data.

In the engineering literature, it is often the case that one drops the $\gamma u^* v$ term in the first equation of system~\eqref{EQ:SHG}. In this case, the system is only one-way coupled. The solution to the first equation only appears in the second equation as the source term. In such a case, well-posedness of the system can be easily established for general boundary conditions. The corresponding inverse problems are also simplified. We will comment more on this issue in the next sections.

\subsection{First- and higher-order linearizations}
\label{SUBSEC:Linearization}

To deal with the challenge caused by the nonlinearity of the forward model~\eqref{EQ:SHG}, we use the technique of linearization~\cite{AsZh-JDE21,Choulli-arXiv2022,Isakov-ARMA93,Isakov-Book06,LaLiLiSa-JMPA21,LuZh-arXiv23,KrUh-arXiv19,KrUh-PAMS20,HaLi-NA23,Kian-Nonlinearity23}. We now document the linearization process.

For a given small number $\eps>0$, let $(u_\eps, v_\eps)$ be the solution to the system
\begin{equation}\label{EQ:SHG Eps}
	\begin{array}{rcll}
  	\Delta u_\eps+ k^2 (1+\eta) u_\eps +ik\sigma u_\eps&=& -k^2 \gamma u_\eps^* v_\eps,
  	&\mbox{in}\ \ \Omega\\
  	\Delta v_\eps+ (2k)^2 (1+\eta) v_\eps +i2k\sigma v_\eps &=& -(2k)^2 \gamma u_\eps^2,
  	&\mbox{in}\ \ \Omega
	\end{array}
\end{equation}
with boundary conditions 
\begin{equation}\label{EQ:BCs One Parameter Family}
	u_\eps = \eps g_1+\dfrac{1}{2}\eps^2 g_2, \qquad v_\eps = \eps h_1+\dfrac{1}{2}\eps^2 h_2,\qquad \mbox{on}\ \ \partial\Omega\,.
\end{equation}
We denote by $(u_0, v_0)=(0, 0)$ the solution for the case of $\eps=0$, and by $H_\eps$ be the data of the form~\eqref{EQ:Data Form} corresponding to $(u_\eps, v_\eps)$, that is
\begin{align*}
    H_\eps = \Gamma \sigma (|u_\eps|^2 + |v_\eps|^2)\,.
\end{align*}

We expect that the solution $(u_\eps, v_\eps)$ varies sufficiently smoothly with respect to $\eps$ when $\eps$ is adequately small. Therefore, formally we have expansions of the solution and the data in the form of:
\begin{equation}\label{EQ:Asymptotic Expansion}
\begin{aligned}
    u_\eps(\bx) &= \eps u^{(1)}(\bx) + \frac12 \eps^2 u^{(2)}(\bx) + o(\eps^2),\\
    v_\eps(\bx) &= \eps v^{(1)}(\bx) + \frac12 \eps^2 v^{(2)}(\bx) + o(\eps^2),\\
    H_\eps(\bx) &= \eps H^{(1)}(\bx) + \frac12 \eps^2 H^{(2)}(\bx) + \frac16 \eps^3 H^{(3)}(\bx) + o(\eps^3),
\end{aligned}
\end{equation}
as $\eps\rightarrow 0$. When this expansion is well-defined, we have that
\begin{align}\label{EQ:Derivatives}
    \begin{split}
    u^{(1)}(\bx) &:= \lim_{\eps\rightarrow 0} \frac{u_\eps(\bx)}{\eps},\qquad u^{(2)}(\bx) := \lim_{\eps\rightarrow 0} \frac{u_\eps(\bx) - \eps u^{(1)}(\bx)}{\frac12 \eps^2},\\
    v^{(1)}(\bx) &:= \lim_{\eps\rightarrow 0} \frac{v_\eps(\bx)}{\eps},\qquad v^{(2)}(\bx) := \lim_{\eps\rightarrow 0} \frac{v_\eps(\bx) - \eps v^{(1)}(\bx)}{\frac12 \eps^2},\\
    H^{(1)}(\bx) &:= \lim_{\eps\rightarrow 0} \frac{H_\eps(\bx)}{\eps},\qquad H^{(2)}(\bx) := \lim_{\eps\rightarrow 0} \frac{H_\eps(\bx) - \eps H^{(1)}(\bx)}{\frac12 \eps^2},\\
    H^{(3)}(\bx) &:= \lim_{\eps\rightarrow 0} \frac{H_\eps(\bx) - \eps H^{(1)}(\bx) - \frac12\eps^2 H^{(2)}(\bx)}{\frac16 \eps^3}.
    \end{split}
\end{align}




Assuming for the moment that all the derivatives are well-defined, straightforward formal calculations then show that on the first order, we have that $(u^{(1)}, v^{(1)})$ solves the boundary value problem:
\begin{equation}\label{EQ:Order eps}
	\begin{array}{rcll}
  	\Delta u^{(1)}+ k^2 (1+\eta) u^{(1)} +ik\sigma u^{(1)} &=& 0,
  	&\mbox{in}\ \ \Omega\\
  	\Delta v^{(1)}+ (2k)^2 (1+\eta) v^{(1)} +i2k\sigma v^{(1)} &=& 0,
  	&\mbox{in}\ \ \Omega\\
	u^{(1)} = g_1, \qquad v^{(1)} = h_1, & & & \mbox{on}\  \partial\Omega
	\end{array}
\end{equation}
while $H^{(1)}$ satisfies
\begin{equation}\label{EQ:Data Order eps}
	H^{(1)}=0\,.
\end{equation}



On the second order, we can formally verify that $(u^{(2)}, v^{(2)})$ solves the boundary value problem:
\begin{equation}\label{EQ:Order eps2}
	\begin{array}{rcll}
  	\Delta u^{(2)}+ k^2 (1+\eta) u^{(2)}+ik\sigma u^{(2)} &=& -2 k^2 \gamma u^{(1)*}v^{(1)},
  	&\mbox{in}\ \ \Omega\\
  	\Delta v^{(2)}+ (2k)^2 (1+\eta) v^{(2)}+i2k\sigma v^{(2)} &=& -2 (2k)^2\gamma (u^{(1)})^2,
  	&\mbox{in}\ \ \Omega\\
   u^{(2)} = g_2, && v^{(2)} = h_2, & \mbox{on}\  \partial\Omega\,.
	\end{array}
\end{equation}
The corresponding perturbative data $H^{(2)}$ can be expressed as
\begin{equation}\label{EQ:Data Order eps2}
	H^{(2)}=2\Gamma \sigma ( u^{(1)*} u^{(1)} + v^{(1)*} v^{(1)})\,.
\end{equation}

A little more algebra shows that the third-order data perturbation is in the form:
\begin{align}\label{EQ:Data Order eps3}
    H^{(3)} = 3\Gamma\sigma\left(u^{(1)*}u^{(2)}+u^{(1)}u^{(2)*}+v^{(1)*}v^{(2)}+v^{(1)}v^{(2)*}\right)\,.
\end{align}


The whole linearization process can be justified mathematically. We summarize the result here.
\begin{theorem}\label{THM:Linearization}
Let $\alpha\in (0,1)$ and $g_1,g_2\in \cC^{2,\alpha}(\partial\Omega;\bbC)$. For sufficiently small $\eps$, let $(u_\eps, v_\eps)$ denote the unique small solution in $\cC^{2,\alpha}(\overline{\Omega};\bbC)\times \cC^{2,\alpha}(\overline{\Omega};\bbC)$ to the system~\eqref{EQ:SHG Eps}. Then the derivatives~\eqref{EQ:Derivatives} are all well-defined. Moreover,~\eqref{EQ:Order eps},~\eqref{EQ:Data Order eps},~\eqref{EQ:Order eps2},~\eqref{EQ:Data Order eps2}, and~\eqref{EQ:Data Order eps3} hold.
\end{theorem}
The proof of this differentiability result is provided in Appendix~\ref{APP:Derivation of the linearization}.

The multilinearization procedure we outlined here is quite standard. It has been improved by many authors and utilized to solve various inverse problems to nonlinear models; see, for instance,~\cite{FeLiLi-arXiv21,KrUh-arXiv19,LaLi-NA22,LaReZh-SIAM22,UhZh-JMPA21} and references therein for some examples of such results.

\section{The reconstruction of \texorpdfstring{($\Gamma$, $\sigma$, $\eta$)}{}}
\label{SEC:IP Step I}

The first inverse problem is therefore to reconstruct $(\Gamma,\sigma,\eta)$ from the data $H^{(2)}$ in~\eqref{EQ:Data Order eps2} with the model for $u^{(1)}$ and $v^{(1)}$ given in~\eqref{EQ:Order eps}. By taking $h_1=0$, the problem reduces to reconstructing $(\Gamma,\sigma,\eta)$ from the data 
\begin{equation}
	H^{(2)}=2\Gamma \sigma u^{(1)*} u^{(1)},
\end{equation}
with the model for $u^{(1)}$ given in \eqref{EQ:Order eps}.

When $\Gamma$ and $\eta$ are known, this problem was analyzed in~\cite{AmGaJiNg-ARMA12,BaReUhZh-IP11}. It was shown that $\sigma$ can be uniquely recovered with a fixed point iteration. More precisely, for the model
\begin{equation}\label{EQ:Helmholtz Scalar}
	\begin{array}{rcll}
		\Delta u + k^2(1+\eta(\bx)) u + ik\sigma(\bx) u &= &0, & \mbox{in}\ \Omega\\
		u & = & g, & \mbox{on}\ \partial\Omega
	\end{array}
\end{equation}
with internal data
\begin{equation}\label{EQ:Helmholtz Data}
	H(\bx)=\Gamma \sigma(\bx)|u|^2\,,
\end{equation}
it is shown in~\cite{BaReUhZh-IP11} that when $\Gamma$ and $\eta$ are known, one can reconstruct $\sigma$ uniquely and stably (in appropriate metrics) from one dataset, provided that the boundary illumination $g$ is appropriately chosen. (More specifically, the proof requires that $g$ is sufficiently close to a a function of the form $e^{\rho\cdot x}|_{\partial\Omega}$, for some $\rho\in\bbC^n$ with $\rho\cdot\rho=0$ and $|\rho|$ sufficiently large.)

In~\cite{AmGaJiNg-ARMA12}, an explicit procedure for reconstructing $\sigma$ is given (again, assuming that $\eta$ and $\Gamma$ are known). Here, we modify the method in order to deal with the case of unknown refractive index $\eta$ and unknown Gr{\"u}neisen coefficient $\Gamma$. We use the procedure to develop a uniqueness and stability result from two well-chosen datasets. 

Let $g_1$ and $g_2$ be two incident sources. We measure data corresponding to the illuminations $g_1+g_2$ and $g_1+ig_2$ in addition to those corresponding to $g_1$ and $g_2$. The linearity of~\eqref{EQ:Helmholtz Scalar} means that solutions corresponding to $g_1+g_2$ and $g_1+ig_2$ are $u_1+u_2$ and $u_1+iu_2$ respectively. The corresponding data are $\Gamma\sigma |u_1+u_2|^2$ and $\Gamma\sigma |u_1+iu_2|^2$ respectively. 
We may now apply the \emph{polarization identity} to get:
\begin{align*}
    u_1 u_2^* = \frac12\left(|u_1+u_2|^2+i|u_1+iu_2|^2-(1+i)|u_1|^2-(1+i)|u_2|^2\right)
\end{align*}
on the inner product space $\bbC$. This gives us that the quantity
\begin{align*}
    \Gamma\sigma u_1 u_2^* =  \frac12\left(\Gamma\sigma|u_1+u_2|^2+i\Gamma\sigma|u_1+iu_2|^2-(1+i)\Gamma\sigma|u_1|^2-(1+i)\Gamma\sigma|u_2|^2\right)
\end{align*}
is known.

Henceforth, given illuminations $\{g_j\}_{j=1}^{2}$, we can reconstruct from the measured internal data the new data:
\begin{equation}\label{EQ:Data Polarized}
    E_j = \Gamma\sigma u_j u_1^*,
\end{equation}
for $j=1,2$.

The above construction can be used to develop a uniqueness result straightforwardly.
\begin{theorem}\label{THM:Uniqueness}
Let $\{g_j\}_{j=1}^{2}$ be a set of incident source functions, and suppose that the measured data $\{E_j\}_{j=1}^{2}$ satisfy the following two conditions:
\begin{enumerate}[label=($\cB$-\roman*)]
    \item $E_1(\bx)\ge \alpha_0>0$ for some $\alpha_0$, a.e. $\bx\in \Omega$.
    \item The vector field
    \[
        \beta(\bx):=\nabla \frac{E_2}{E_1}\,
    \]
    is at least $W^{1,\infty}$, and $|\beta|\ge \beta_0>0$ for some $\beta_0$, a.e. $\bx\in\Omega$.
\end{enumerate}
Then $\Gamma$, $\eta$, and $\sigma$ are uniquely determined from the data $\{E_j\}_{j=1}^{2}$.
\end{theorem}
\begin{proof} 
We follow the procedures developed in~\cite{AmGaJiNg-ARMA12,BaRe-IP11}. We multiply the equation for $u_1$ by $u_2$ and multiply the equation for $u_2$ by $u_1$. We subtract the results to have
\[
    u_1\Delta u_2-u_2\Delta u_1=0\,.
\]
We can then rewrite this into 
\[
    \nabla \cdot u_1^2 \nabla \frac{u_2}{u_1} = \nabla \cdot u_1^2 \nabla \frac{E_2}{E_1} =\nabla\cdot (u_1^2 \beta)=0\,.
\]
The vector field $\beta$ is known from the data. Therefore the above equation is a transport equation for $u_1$, that is,
\begin{equation}\label{EQ:Transport}
    \nabla\cdot (u_1^2 \beta)=0,\ \ \mbox{in}\ \  \Omega, \qquad u_1=g_1,\ \ \mbox{on}\ \ \partial\Omega\,.
\end{equation}
With the assumption in ($\cB$-$ii$), classical results in~\cite{Ambrosio-IM04,BoCr-SIAM06,CoLe-DMJ02,DiLi-AM89} show that there exists a unique weak solution $u_1$ to~\eqref{EQ:Transport}. This gives us the unique reconstruction of $u_1$. 

Now that we have reconstructed $u_1$, we can use the equation~\eqref{EQ:Helmholtz Scalar} to reconstruct the potential $q$:
\begin{equation}\label{EQ:Potential}
    q(\bx):=k^2(1+\eta) + ik\sigma = -\frac{\Delta u_1}{u_1}\,.
\end{equation}
This gives us $\eta$ and $\sigma$ (which are obtained by taking real and imaginary parts of $q$). The last step is to reconstruct $\Gamma$ as 
\begin{equation}\label{EQ:Gamma}
    \Gamma=\frac{H_1}{\sigma|u_1|^2}\,.
\end{equation}
The proof is complete.
\end{proof}
This uniqueness result shows a dramatic difference between the inverse problem defined by ~\eqref{EQ:Helmholtz Scalar} and~\eqref{EQ:Helmholtz Data} and a similar inverse problem in quantitative photoacoustic tomography in~\cite{BaRe-IP11} where it is show that the multiplicative coefficient $\Gamma$ causes non-uniqueness in the reconstructions, independent of the amount of data available.

The proof of the above uniqueness result is constructive in the sense that it provides an explicit way to solve the inverse problem: solving~\eqref{EQ:Transport} for $u_1$, computing $q$ using~\eqref{EQ:Potential} and then computing $\Gamma$ as in~\eqref{EQ:Gamma}. 

In fact, the above explicit reconstruction procedure also leads to partial (weighted) stability results for the inverse problem.
\begin{theorem}
Let $E=(E_1,E_2)$ and $\wt E=(\wt E_1, \wt E_2)$ be the data corresponding to the coefficients $(\Gamma, \eta, \sigma)\in\cY^3$ and $(\wt\Gamma, \wt \eta, \wt \sigma)\in\cY^3$ respectively, generated from illumination source pair $(g_1, g_2)$. Under the assumption that $E$ and $\wt E$ satisfy ($\cB$-i)-($\cB$-ii), we assume further that $g_1$ and $g_2$ are selected such that $\frac{E_2}{E_1}$ is sufficiently small. Then we have that, for some constants $c>0$,
\begin{equation}\label{EQ:Stab}
   \|\Gamma\sigma-\wt\Gamma\wt\sigma\|_{L^2(\Omega)} \le c\Big(\|H_1-\wt H_1\|_{L^2(\Omega)}+ \|E-\wt E\|_{(\cH^2(\Omega))^2} \Big)\,.
\end{equation}
\end{theorem}
\begin{proof}
We first observe that
\[
   \Gamma \sigma-\wt\Gamma\wt\sigma=\frac{H_1}{|u_1|^2}-\frac{\wt H_1}{|\wt u_1|^2} =\frac{H_1-\wt H_1}{|u_1|^2}+\frac{\wt H_1}{|\wt u_1|^2} \frac{|\wt u_1|^2-|u_1|^2}{|u_1|^2}\,.
\]
This, together with the Triangle Inequality and the fact that $|u_k|^2$ is bounded from below, gives us
\begin{equation}\label{EQ:Gamma-Sigma}
   \|\Gamma \sigma-\wt\Gamma\wt\sigma\|_{L^2(\Omega)}\le \wt c \Big(\|H_1-\wt H_1\|_{L^2(\Omega)} + \|\wt u_1^2-u_1^2\|_{L^2(\Omega)}\Big)\,,
\end{equation}
for some $\wt c>0$.

To bound the second term in~\eqref{EQ:Gamma-Sigma} by the data, let $\xi=u_1^2$ and $\wt \xi=\wt u_1^2$. Then we have from the equations $\nabla\cdot\xi\beta=0$ and $\nabla\cdot\wt\xi\wt\beta=0$ that
\begin{equation}\label{EQ:Transport Difference}
    \nabla \cdot \Big((\xi-\wt\xi) \beta\Big) + \nabla\cdot \Big(\wt \xi (\beta-\wt\beta)\Big) =0\,.
\end{equation}
This can be further rewritten into
\begin{equation*}
	\beta\cdot\nabla(\xi-\wt\xi)=-(\xi-\wt\xi)\nabla\cdot\beta - \nabla\cdot \Big(\wt \xi (\beta-\wt\beta)\Big)\,,
\end{equation*}
which immediately leads to the bound
\begin{equation}\label{EQ:ubound-1}
	\|\beta\cdot\nabla(\xi-\wt\xi)\|_{L^2(\Omega)} \le \|(\xi-\wt\xi)\nabla\cdot\beta\|_{L^2(\Omega)}^2+\|\nabla\cdot \Big(\xi (\beta-\wt\beta)\Big)\|_{L^2(\Omega)}\,.
\end{equation}
With the same algebra, we can derive the bound
\begin{equation}\label{EQ:ubound-2}
	\|\wt \beta\cdot\nabla(\xi-\wt\xi)\|_{L^2(\Omega)}\le \|(\xi-\wt\xi)\nabla\cdot\wt \beta\|_{L^2(\Omega)}+\|\nabla\cdot \Big(\wt \xi (\beta-\wt\beta)\Big)\|_{L^2(\Omega)}\,.
\end{equation}

We now multiply~\eqref{EQ:Transport Difference} by $(\xi-\wt\xi)^*$ to have the equation, after a little algebra,
\begin{displaymath} 
\nabla\cdot \Big(\big|\xi-\wt\xi\big|^2 \beta\Big) - (\xi-\wt\xi)\beta\cdot \nabla (\xi-\wt\xi)^* + (\xi-\wt\xi)^*\nabla\cdot  \wt\xi (\beta-\wt\beta) =0\,.
\end{displaymath}
Integrating this equation against a test function $\phi\in \cH^1(\Omega)$ and using integration-by-parts on the last term lead us to the identity
\begin{multline*} 
\int_\Omega \big|\xi-\wt\xi\big|^2 \beta\cdot\nabla \phi d\bx + \int_\Omega (\xi-\wt\xi)\beta \cdot \phi\nabla (\xi-\wt\xi)^* d\bx \\ 
+\int_\Omega (\xi-\wt\xi)^* \wt\xi (\beta-\wt\beta) \cdot \nabla \phi d\bx +\int_\Omega \phi \wt\xi (\beta-\wt\beta)\cdot\nabla (\xi-\wt\xi)^* d\bx =0\,.
\end{multline*}
To simplify the presentation, we combine the second and the fourth terms in the equation to have
\begin{multline*} 
\int_\Omega \big|\xi-\wt\xi\big|^2 \beta\cdot\nabla \phi d\bx + \int_\Omega (\beta \xi-\wt\beta\wt\xi)\cdot \phi\nabla (\xi-\wt\xi)^* d\bx +\int_\Omega (\xi-\wt\xi)^* \wt\xi (\beta-\wt\beta) \cdot \nabla \phi d\bx  =0\,.
\end{multline*}
Taking the test function $\phi=\frac{E_2^*}{E_1^*}$ (hence $\nabla \phi=\nabla\frac{E_2^*}{E_1^*}=\beta^*$), we have
\begin{equation*} 
\int_\Omega \big|\xi-\wt\xi\big|^2 |\beta|^2 d\bx = - \int_\Omega (\beta \xi-\wt\beta\wt\xi)\cdot \frac{E_2^*}{E_1^*}\nabla (\xi-\wt\xi)^* d\bx -\int_\Omega (\xi-\wt\xi)^*\beta^* \cdot \wt\xi (\beta-\wt\beta)  d\bx  \,.
\end{equation*}
This gives us the bound
\begin{equation}\label{EQ:ubound-3}
\|(\xi-\wt\xi)\beta\|_{L^2(\Omega)}^2\le  \int_\Omega \big|(\beta \xi-\wt\beta\wt\xi)\cdot \frac{E_2^*}{E_1^*}\nabla (\xi-\wt\xi)^*\big| d\bx + \int_\Omega \big|(\xi-\wt\xi)^*\beta^* \cdot \wt\xi (\beta-\wt\beta)\big|  d\bx  \,.
\end{equation}

The first term on the right-hand side of~\eqref{EQ:ubound-3} can be bounded as follows:
\begin{multline*}
 \int_\Omega \big|(\beta \xi-\wt\beta\wt\xi)\cdot \frac{E_2^*}{E_1^*}\nabla (\xi-\wt\xi)^*\big| d\bx 
 \le \|\frac{E_2^*}{E_1^*}\|_{L^\infty(\Omega)} \int_\Omega \big|(\beta \xi-\wt\beta\wt\xi)\cdot \nabla (\xi-\wt\xi)^*\big| d\bx \\ 
 \le \frac{1}{2}\|\frac{E_2^*}{E_1^*}\|_{L^\infty(\Omega)}\Big[ \int_\Omega \big|\beta \xi\cdot \nabla (\xi-\wt\xi)^*\big|^2 d\bx +\int_\Omega \big|\wt\beta \wt \xi\cdot \nabla (\xi-\wt\xi)^*\big|^2 d\bx\Big] \\ 
\le 2\|\frac{E_2^*}{E_1^*}\|_{L^\infty(\Omega)}\Big[\|(\xi-\wt\xi)\nabla\cdot\beta\|_{L^2(\Omega)}^2+\|\nabla\cdot \Big(\xi (\beta-\wt\beta)\Big)\|_{L^2(\Omega)}^2\Big]\,,
\end{multline*}
where we have used~\eqref{EQ:ubound-1} and~\eqref{EQ:ubound-2} to get the last inequality. The second term on the right-hand side of~\eqref{EQ:ubound-3} can be bounded as:
\[
     \int_\Omega \big|(\xi-\wt\xi)^*\beta^* \cdot \wt\xi (\beta-\wt\beta)\big|  d\bx \le \frac{1}{2}\Big[\frac{1}{\kappa^2}\|(\xi-\wt\xi)\beta\|_{L^2(\Omega)}^2 +\kappa^2 \|(\beta-\wt\beta)\wt\xi\|_{L^2(\Omega)}^2\Big]\,,
\]
for any $\kappa>0$.
 
Under the assumption that $|\frac{E_2}{E_1}|$ is sufficiently small, we can take $\kappa$ to be sufficiently large so that ~\eqref{EQ:ubound-3} now implies that
\begin{equation}\label{EQ:ubound-4}
	\|\xi-\wt\xi \|_{L^2(\Omega)}^2\lesssim  \|(\beta-\wt\beta)\wt\xi\|_{L^2(\Omega)}^2 +\|\nabla\cdot \Big(\xi (\beta-\wt\beta)\Big)\|_{L^2(\Omega)}^2\,.
\end{equation}

The next step is to bound $\|\beta-\wt\beta\|_{L^2(\Omega)}$ and $\|\nabla\cdot \xi \big(\beta-\wt\beta\big)\|_{L^2(\Omega)}$. To this end, we use the expansion
\begin{multline*}
\beta-\wt\beta=(\wt E_1-E_1)\nabla \frac{\wt E_2}{E_1\wt E_1}+\frac{\wt E_2}{E_1\wt E_1}\nabla(\wt E_1-E_1)\\
+(E_2-\wt E_2)\nabla\frac{1}{E_1}+\frac{1}{E_1^2}\nabla(E_2-\wt E_2)+\frac{1}{E_1\wt E_1}(\wt E_1-E_1)\nabla(E_2-\wt E_2)\,,
\end{multline*}
to derive the bound
\[
    \|\beta-\wt \beta\|_{L^2(\Omega)}\lesssim \|E-\wt E\|_{(\cH^1(\Omega))^2}\,.
\]
In a similar manner, we find that 
\[
    \|\nabla\cdot \xi \big(\beta-\wt \beta\big)\|_{L^2(\Omega)}\lesssim \|E-\wt E\|_{(\cH^2(\Omega))^2}\,.
\]
Plugging these results into~\eqref{EQ:ubound-4} will give us
\[
    \|\xi-\wt\xi\|_{L^2(\Omega)} \lesssim \|E-\wt E\|_{(\cH^2(\Omega))^2}\,.
\]
This, together with the bound~\eqref{EQ:Gamma-Sigma}, will lead us to the stability results of ~\eqref{EQ:Stab}.
\end{proof}

\begin{remark} 
 With the standard techniques complex geometrical solutions, one can show that for every value of the true coefficients $\eta, \sigma \in \cH^m(\Omega)$, where $m>1+\frac{d}{2}$, there exists a set of illuminations $(g_j)_{j=1}^{d+1}$ such that the corresponding measured data $(E_j)_{j=1}^{d+1}$ satisfies both conditions ($\cB$-$i$) and ($\cB$-$ii$)~\cite{AmGaJiNg-ARMA12}. 
 In fact, following \cite{Alberti-arXiv22}, it may be possible to ensure that ($\cB$-$i$) and ($\cB$-$ii$) hold with high probability by drawing the boundary illuminations $(g_j)_{j=1}^{d+1}$ independently at random from a sub-Gaussian distribution on $\cH^{1/2}(\partial\Omega)$. 
\end{remark}

\begin{remark} 
We observe that the above reconstruction procedure also works in the case when the internal datum is of the form $H=|u|$ (in which case $|u|^2$ is known), that is, the datum is independent of $\Gamma\sigma$.
\end{remark}

\section{The reconstruction of \texorpdfstring{$\gamma$}{}}
\label{SEC:IP Step II}

The remaining problem is to reconstruct $\gamma$ using third-order perturbation of the data. In the rest of this section, we assume that in addition to the internal data~\eqref{EQ:Data}, we also have access to the Dirichlet-to-Neumann map
\begin{equation}\label{EQ:DtN}
    \Pi_{\gamma}: (g, h) \mapsto \left(\pdr{u}{\nu}\Big|_{\partial\Omega}, \pdr{v}{\nu}\Big|_{\partial\Omega}\right)\equiv (J_u,J_v)\,.
\end{equation}
Note that we omit the dependence of $\Pi$ on $\Gamma$, $\eta$, and $\sigma$ intentionally here since those coefficients are already known. 

The multilinearization of $(J_u, J_v)$ can be established with the calculations in Appendix~\eqref{APP:Derivation of the linearization}. We will directly use the derivatives $(J_u^{(1)}, J_v^{(1)})$  and $(J_u^{(2)}, J_v^{(2)})$.

Let us recall that the third-order derivative of the data $H^{(3)}$ is given in~\eqref{EQ:Data Order eps3}. This implies that
\begin{align*}    u^{(1)*}u^{(2)}+u^{(1)}u^{(2)*}+v^{(1)*}v^{(2)}+v^{(1)}v^{(2)*} = \frac{H^{(3)}}{3\Gamma\sigma},
\end{align*}
where $(u^{(1)}, v^{(1)})$ and $(u^{(2)}, v^{(2)})$ are respectively the solutions to~\eqref{EQ:Order eps} and~\eqref{EQ:Order eps2}, is known in $\Omega$. 

From now on, we set $g_2=h_2=0$ in \eqref{EQ:BCs One Parameter Family}. Consequently, the system \eqref{EQ:Order eps2} for $(u^{(2)}, v^{(2)})$ reduces to 
\begin{equation}\label{EQ:Order eps2 ZeroBC}
	\begin{array}{rcll}
  	\Delta u^{(2)}+ k^2 (1+\eta) u^{(2)}+ik\sigma u^{(2)} &=& -2 k^2 \gamma u^{(1)*}v^{(1)},
  	&\mbox{in}\ \ \Omega\\
  	\Delta v^{(2)}+ (2k)^2 (1+\eta) v^{(2)}+i2k\sigma v^{(2)} &=& -2 (2k)^2\gamma (u^{(1)})^2,
  	&\mbox{in}\ \ \Omega\\
	u^{(2)} = 0, && v^{(2)} = 0,& \mbox{on}\ \ \partial\Omega
    \end{array}
\end{equation}
We can now take the complex conjugate of~\eqref{EQ:Order eps2 ZeroBC} and leverage the fact that $\gamma$ is real-valued to write down the following system of \emph{linear} equations for $(u^{(2)}, v^{(2)})$, $(u^{(2)*}, v^{(2)*})$, and $\gamma$:
\begin{equation}\label{EQ:Lin Sys}
\begin{array}{rcll}
(\Delta + q_1)u^{(2)} + 2k^2 u^{(1)*}v^{(1)}\gamma &=&0,&\mbox{in}\ \Omega\\
    (\Delta + q_1^*)u^{(2)*} + 2k^2 u^{(1)}v^{(1)*}\gamma &=& 0,&\mbox{in}\ \Omega\\
    (\Delta + q_2)v^{(2)} + 2(2k)^2(u^{(1)})^2\gamma &=& 0, &\mbox{in}\ \Omega\\
    (\Delta + q_2^*)v^{(2)*} + 2(2k)^2(u^{(1)*})^2\gamma &=& 0, &\mbox{in}\ \Omega\\
u^{(1)*}u^{(2)}+u^{(1)}u^{(2)*}+v^{(1)*}v^{(2)}+v^{(1)}v^{(2)*} &=& \frac{H^{(3)}}{3\Gamma\sigma}, &\mbox{in}\ \Omega\\
    (u^{(2)},u^{(2)*},v^{(2)},v^{(2)*}) &= &(0,0,0,0),&\mbox{on}\ \partial\Omega\\
    \left(\pdr{u^{(2)}}{\nu},\pdr{u^{(2)*}}{\nu},\pdr{v^{(2)}}{\nu},\pdr{v^{(2)*}}{\nu}\right) &=& (J_u^{(2)},J_u^{(2)*},J_v^{(2)},J_v^{(2)*}),&\mbox{on}\ \partial\Omega\,,
\end{array}    
\end{equation}
where we have used the notation
\begin{align*}
q_1 := k^2(1+\eta) + ik\sigma,\qquad q_2 := (2k)^2(1+\eta) + i2k\sigma\,.
\end{align*}

If we can solve~\eqref{EQ:Lin Sys}, we can reconstruct $\gamma$ (and the associated $(u^{(2)}, v^{(2)})$). This is a non-iterative reconstruction method. In the rest of this section, we show that $\gamma$ can be uniquely reconstructed from available data by analyzing the uniqueness of the solution to the linear system~\eqref{EQ:Lin Sys}. The analysis is based on the uniqueness theory for redundant elliptic systems reviewed in~\cite{Bal-CM13}, which we summarize briefly in Appendix~\ref{APP:Elliptic system theory} for the convenience of the readers.

\begin{theorem}\label{THM:Reconstruction of gamma} Let $X\subset \bbR^n$ be an open set and $\Gamma$, $\eta$, $\sigma$ be given positive $\cC^2$ functions on $\overline{X}$. For every bounded open subset $\Omega\subset X$ with smooth boundary $\partial\Omega$, we denote by $(\Lambda_\gamma^{(\Omega)}, \Pi_\gamma^{(\Omega)})$ and $(\Lambda_{\wt \gamma}^{(\Omega)}, \Pi_{\wt \gamma}^{(\Omega)})$ the data corresponding to $\gamma\in\cY$ and $\wt\gamma\in\cY$ respectively. 
Let $\bx_0\in X$ be arbitrary. Then there exists $\epsilon=\epsilon(\eta,\sigma,\bx_0)>0$ such that 
for every $\Omega\subset B(\bx_0, \epsilon)$, there exist $g_1$ and $h_1$ in \eqref{EQ:BCs One Parameter Family} such that 
\[
    (H^{(3)}, J_u^{(2)},  J_v^{(2)})=(\wt H^{(3)}, \wt J_u^{(2)}, \wt J_v^{(2)}) \implies \gamma=\wt \gamma\,,
\]
and for all $p>1$, 
\begin{multline}\label{EQ:Stability}
    \|\gamma-\wt{\gamma}\|_{L^p(\Omega)} \le C\Big(\Big\|H^{(3)} - \wt{H}^{(3)}\Big\|_{W^{2,p}(\Omega)} 
    \\+ \Big\|J_u^{(2)} - \wt J_u^{(2)}\Big\|_{W^{1-1/p, p}(\partial\Omega)}
    + \Big\|J_v^{(2)} - \wt J_v^{(2)}\Big\|_{W^{1-1/p, p}(\partial\Omega)} \Big),
\end{multline}
for some constant $C=C(\Omega,\Gamma,\eta,\sigma)>0$.
\end{theorem}

\begin{proof} 
Let us define 
\begin{align*}
    \wh{u} := u - \wt{u},\quad \wh{v} := v - \wt{v},\quad \wh{\gamma} := \gamma - \wt{\gamma},\quad \wh{H} := H - \wt{H}, \quad \wh J_u := J_u - \wt J_u,\quad \wh J_v := J_v - \wt J_v.
\end{align*}
The linear system~\eqref{EQ:Lin Sys} then implies that 
\begin{equation}
\begin{array}{rcll}
    (\Delta + q_1)\wh{u}^{(2)} + 2k^2 u^{(1)*}v^{(1)}\wh{\gamma} &=& 0, \label{EQ:gamma hat}&\mbox{in}\ \ \Omega\\
    (\Delta + q_1^*)\wh{u}^{(2)*} + 2k^2 u^{(1)}v^{(1)*}\wh{\gamma} &=& 0,&\mbox{in}\ \  \Omega\\
    (\Delta + q_2)\wh{v}^{(2)} + 2(2k)^2(u^{(1)})^2\wh{\gamma} &=& 0, &\mbox{in}\ \  \Omega\\
    (\Delta + q_2^*)\wh{v}^{(2)*} + 2(2k)^2(u^{(1)*})^2\wh{\gamma} &=& 0, &\mbox{in}\ \Omega\\
u^{(1)*}\wh{u}^{(2)}+u^{(1)}\wh{u}^{(2)*}+v^{(1)*}\wh{v}^{(2)}+v^{(1)}\wh{v}^{(2)*} &=& \frac{\wh{H}^{(3)}}{3\Gamma\sigma}, &\mbox{in}\  \ \Omega\\(\wh{u}^{(2)},\wh{u}^{(2)*},\wh{v}^{(2)},\wh{v}^{(2)*}) &=& (0,0,0,0),&\mbox{on}\  \  \partial\Omega\\
    \left(\pdr{\wh{u}^{(2)}}{\nu},\pdr{\wh{u}^{(2)*}}{\nu},\pdr{\wh{v}^{(2)}}{\nu},\pdr{\wh{v}^{(2)*}}{\nu}\right) &=& (\wh J_u^{(2)},\wh J_u^{(2)*},\wh J_v^{(2)},\wh J_v^{(2)*}),&\mbox{on}\  \ \partial\Omega\,.
\end{array}
\end{equation}

We first eliminate $\wh{\gamma}$ by plugging in $\wh{\gamma} = -\frac{(\Delta+q_2^*)\wh{v}^{(2)*}}{2(2k)^2(u^{(1)*})^2}$ from the fourth equation into the first three equations. We then take the Laplacian of the last equation. These procedures lead us to the linear system 
\begin{align*}
    (\Delta + q_1)\wh{u}^{(2)} -\frac{v^{(1)}}{4u^{(1)*}}(\Delta + q_2^*)\wh{v}^{(2)*} &= 0,\\
    (\Delta + q_1^*)\wh{u}^{(2)*} - \frac{u^{(1)}v^{(1)*}}{4(u^{(1)*})^2}(\Delta + q_2^*)\wh{v}^{(2)*} &= 0,\\
    (\Delta + q_2)\wh{v}^{(2)} - \frac{(u^{(1)})^2}{(u^{(1)*})^2}(\Delta + q_2^*)\wh{v}^{(2)*} &= 0, \\    \Delta(u^{(1)*}\wh{u}^{(2)})+\Delta(u^{(1)}\wh{u}^{(2)*})+\Delta(v^{(1)*}\wh{v}^{(2)})+\Delta(v^{(1)}\wh{v}^{(2)*}) &= \Delta\left(\frac{\wh{H}^{(3)}}{3\Gamma\sigma}\right),
\end{align*}
in the unknowns $\wh{u}^{(2)}$, $\wh{u}^{(2)*}$, $\wh{v}^{(2)}$, and $\wh{v}^{(2)*}$. The system may be written in the following matrix form, for the quantity $w:=(\wh{u}^{(2)},\wh{u}^{(2)*},\wh{v}^{(2)},\wh{v}^{(2)*})$:
\begin{equation}\label{EQ:Elliptic system}
    \begin{array}{rcll}
    \cA w &=& \cS, & \mbox{in}\ \ \Omega\\
    w&=& (0,0,0,0),&\mbox{on}\ \partial\Omega\\
    \pdr{w}{\nu} &=& (\wh J_u^{(2)},\wh J_u^{(2)*},\wh J_v^{(2)},\wh J_v^{(2)*}),&\mbox{on}\ \partial\Omega\,,
    \end{array}
\end{equation}
where
\begin{align*}
    \cA(\bx,D) := \begin{pmatrix}
    \Delta+q_1 & & & -\frac{v^{(1)}}{4u^{(1)*}}(\Delta + q_2^*)\\
    & \Delta+q_1^*& & -\frac{u^{(1)}v^{(1)*}}{4(u^{(1)*})^2}(\Delta + q_2^*)\\
    & & \Delta+q_2 & -\frac{(u^{(1)})^2}{(u^{(1)*})^2}(\Delta + q_2^*)\\
    \Delta(u^{(1)*}\cdot) & \Delta(u^{(1)}\cdot) & \Delta(v^{(1)*}\cdot) & \Delta(v^{(1)}\cdot)
    \end{pmatrix}\,,
\end{align*}
and
\begin{align*}
S := \left(0,0,0,\Delta\left(\frac{\wh{H}^{(3)}}{3\Gamma\sigma}\right)\right)\,.
\end{align*}

In the rest of the proof, we show that $\cA$ is an elliptic operator in the sense of Douglis and Nirenberg~\cite{Bal-CM13}. For the convenience of the reader, we provide a brief review of elliptic system theory in Appendix~\ref{APP:Elliptic system theory}. We choose the Douglis-Nirenberg numbers 
\begin{align}\label{EQ:Douglis Nirenberg numbers}
    (s_1,s_2,s_3,s_4) = (0,0,0,0),\qquad (t_1,t_2,t_3,t_4) = (2,2,2,2).
\end{align}
The principal part $\cA_0(\bx,D)$ of $\cA$ has symbol 
\begin{align*}
    \cA_0(\bx,\bxi) := \begin{pmatrix}
    |\bxi|^2 & & & -\frac{v^{(1)}}{4u^{(1)*}}|\bxi|^2\\
    & |\bxi|^2& & -\frac{u^{(1)}v^{(1)*}}{4(u^{(1)*})^2}|\bxi|^2\\
    & & |\bxi|^2 & -\frac{(u^{(1)})^2}{(u^{(1)*})^2}|\bxi|^2\\
    u^{(1)*}|\bxi|^2 & u^{(1)}|\bxi|^2 & v^{(1)*}|\bxi|^2 & v^{(1)}|\bxi|^2
    \end{pmatrix}\,.
\end{align*}

One readily sees that $\cA_0(\bx_0,\bxi)$ has full rank $4$ for all $\bxi\neq 0$ if and only if the following condition holds at $\bx_0$: 
\begin{align*}
    -\frac{v^{(1)}}{4u^{(1)*}}\cdot u^{(1)*} -\frac{u^{(1)}v^{(1)*}}{4(u^{(1)*})^2}\cdot u^{(1)} -\frac{(u^{(1)})^2}{(u^{(1)*})^2}\cdot v^{(1)*} \neq v^{(1)},
\end{align*}
or equivalently 
\begin{align*}
    -\frac{(u^{(1)})^2}{(u^{(1)*})^2} \neq \frac{v^{(1)}}{v^{(1)*}}
\end{align*}
at $\bx_0$. This condition on $u^{(1)}(\bx_0)$ and $v^{(1)}(\bx_0)$ is easily achieved by selecting $g_1$ and $h_1$ appropriately. To be precise, let us consider some ball $B(\bx_0,\epsilon_0) \subset X$ and let $u_0$ and $v_0$ be any $\cC^2$ functions on $B(\bx_0,\epsilon_0)$ satisfying
\begin{align*}
    \Delta u_0+ q_1 u_0 &= 0,\quad  \mbox{in}\ B(\bx_0,\epsilon_0), \\
    \Delta v_0 +  q_2 v_0 &= 0,\quad \mbox{in}\ B(\bx_0,\epsilon_0), \\
    -\frac{u_0^2}{(u_0^*)^2}(\bx_0) &\neq \frac{v_0}{v_0^*}(\bx_0).
\end{align*}
The existence of such $u_0$ and $v_0$ is obvious as we can take $u_0$ and $v_0$ to be any solutions to the first two equations and rescale them by a suitable complex constant to satisfy the condition $-\frac{u_0^2}{(u_0^*)^2}(\bx_0) \neq \frac{v_0}{v_0^*}(\bx_0)$.
It is also useful to observe that $u_0$ and $v_0$ depend only on $\eta$ and $\sigma$, not $\gamma$.

Now suppose $\Omega\subset B(\bx_0,\epsilon_0)$, and select $g_1=u_0|_{\partial\Omega}$ and $h_1=v_0|_{\partial\Omega}$ in~\eqref{EQ:Order eps}, so that 
\begin{align*}
u^{(1)}=u_0|_\Omega,\qquad v^{(1)}=v_0|_\Omega\,.
\end{align*}
This means we can write $\cA$ as 
\begin{align*}
    \cA(\bx,D) := \begin{pmatrix}
    \Delta+q_1 & & & -\frac{v_0}{4u_0^*}(\Delta + q_2^*)\\
    & \Delta+q_1^*& & -\frac{u_0 v_0^*}{4(u_0^*)^2}(\Delta + q_2^*)\\
    & & \Delta+q_2 & -\frac{u_0^2}{(u_0^*)^2}(\Delta + q_2^*)\\
    \Delta(u_0^*\cdot) & \Delta(u_0\cdot) & \Delta(v_0^*\cdot) & \Delta(v_0\cdot)
    \end{pmatrix}\,.
\end{align*}
By construction, the constant-coefficient operator $\cA(\bx_0,D)$ with coefficients frozen at $x_0$ is elliptic. Additionally, from the continuity of $u_0$ and $v_0$ we see that there exists $\epsilon_1=\epsilon_1(\eta,\sigma,\bx_0)>0$ such that $-\frac{u_0^2}{(u_0^*)^2}\neq \frac{v_0}{v_0^*}$ on $B(\bx_0,\epsilon_1)$. That is, for every $\Omega\subset B(\bx_0,\epsilon_1)$, the operator $\cA(\bx,D)$ is elliptic on $\Omega$. 

Moreover, observe that the Douglis-Nirenberg numbers~\eqref{EQ:Douglis Nirenberg numbers} of $\cA$ satisfy $s_i=0$ for all $i$ and $t_j = 2$ is independent of $j$. This means that the uniqueness theory for elliptic systems presented in~\cite[Section 3]{Bal-CM13} applies. More specifically, by Theorem~\ref{THM:Small Domain Uniqueness}, we conclude that there exists $\epsilon_2=\epsilon_2(\eta,\sigma,\bx_0)>0$ such that for every $\Omega\subset B(\bx_0,\epsilon_2)$, the boundary value problem \eqref{EQ:Elliptic system} has a unique solution. 

Now set $\epsilon = \min\{\epsilon_1,\epsilon_2\}>0$ and let $\Omega\subset B(\bx_0,\epsilon)$, so that $\cA$ is an elliptic operator on $\Omega$ and the problem \eqref{EQ:Elliptic system} has a unique solution.

By elliptic regularity estimate Theorem~\ref{THM:Elliptic Dirichlet Stability Estimate} applied to \eqref{EQ:Elliptic system} with $\ell=0$, there exist constants $C=C(\Omega,\eta,\sigma,u_0) = C(\Omega,\eta,\sigma)$ and $C_2=C_2(\Omega,\eta,\sigma)$ such that 
\begin{align*}
    \|\wh{u}^{(2)}\|_{W^{2,p}(\Omega)} + \|\wh{v}^{(2)}\|_{W^{2,p}(\Omega)}&\le C\Bigg(\left\|\Delta\left(\frac{\wh{H}^{(3)}}{3\Gamma\sigma}\right)\right\|_{L^p(\Omega)} 
    + \left\|\wh J_u^{(2)}\right\|_{W^{1-1/p,p}(\partial\Omega)}\\
    &\qquad\qquad + \left\|\wh J_v^{(2)}\right\|_{W^{1-1/p,p}(\partial\Omega)}\Bigg) + C_2(\|\wh{u}^{(2)}\|_{L^p(\Omega)} + \|\wh{v}^{(2)}\|_{L^p(\Omega)}).
\end{align*}
In fact, since the solution is unique we may drop the last term on the right-hand side, i.e. set $C_2=0$. This gives 
\begin{align*}
    &\|\wh{u}^{(2)}\|_{W^{2,p}(\Omega)} + \|\wh{v}^{(2)}\|_{W^{2,p}(\Omega)}\\
    &\le C\left(\left\|\Delta\left(\frac{\wh{H}^{(3)}}{3\Gamma\sigma}\right)\right\|_{L^p(\Omega)} 
    + \left\|\wh J_u^{(2)}\right\|_{W^{1-1/p,p}(\partial\Omega)} + \left\|\wh J_v^{(2)}\right\|_{W^{1-1/p,p}(\partial\Omega)}\right)\\
    &\le C\left(\|\wh{H}^{(3)}\|_{W^{2,p}(\Omega)} 
    + \left\|\wh J_u^{(2)}\right\|_{W^{1-1/p,p}(\partial\Omega)} + \left\|\wh J_v^{(2)}\right\|_{W^{1-1/p,p}(\partial\Omega)}\right),
\end{align*}
where $C=C(\Omega,\Gamma,\eta,\sigma)$.

Finally we substitute this estimate into~\eqref{EQ:gamma hat} to obtain 
\begin{align*}
    \|\wh{\gamma}\|_{L^p(\Omega)} &= \left\|-\frac{(\Delta + q_1)\wh{u}^{(2)}}{2k^2 u_0^* v_0}\right\|_{L^p(\Omega)} \le C\|\wh{u}^{(2)}\|_{W^{2,p}(\Omega)}\\
    &\le C\left(\|\wh{H}^{(3)}\|_{W^{2,p}(\Omega)} 
    + \left\|\wh J_u^{(2)}\right\|_{W^{1-1/p,p}(\partial\Omega)} + \left\|\wh J_v^{(2)}\right\|_{W^{1-1/p,p}(\partial\Omega)}\right),
\end{align*}
where once again $C=C(\Omega,\Gamma,\eta,\sigma)$. This is precisely the desired stability estimate \eqref{EQ:Stability}.
\end{proof}

The above theory on the reconstruction of the coefficient $\gamma$ requires both the availability of the additional boundary data~\eqref{EQ:DtN} and the assumption that $\Omega$ is sufficiently small. We made these assumptions merely to simplify the proof. We believe that they can be removed without breaking the uniqueness and stability results.

\section{Numerical experiments}
\label{SEC:Numerics}

We now present some numerical simulations to demonstrate the quality of reconstructions that can be achieved for the inverse problem. 

We will perform numerical reconstructions with a slightly simplified version of model~\eqref{EQ:SHG}:
\begin{equation}\label{EQ:SHG Robin One Way Coupled}
	\begin{array}{rcll}
  	\Delta u+ k^2 (1+\eta) u +ik\sigma u&=& 0, & \mbox{in}\ \ \Omega\\
  	\Delta v+ (2k)^2 (1+\eta) v +i2k\sigma v &=& -(2k)^2 \gamma u^2,  &\mbox{in}\ \ \Omega\\
	u = g, \qquad v+i2k\bnu\cdot \nabla v &=& 0, & \mbox{on}\ \partial\Omega
	\end{array}
\end{equation}
In other words, we omit the backward coupling term on the right-hand side of the first equation in~\eqref{EQ:SHG}. 
This model is connected to the linearized problem in~\eqref{EQ:Order eps2}. Indeed, if we take the boundary condition of $v^{(1)}$ to be $0$, that is, $h_1=0$ in~\eqref{EQ:Order eps}, then the first equation in~\eqref{EQ:Order eps} and the second equation in~\eqref{EQ:Order eps2} can be combined to get~\eqref{EQ:SHG Robin One Way Coupled}.
Note that we intentionally changed the Dirichlet boundary condition for $v$ to the more realistic Robin boundary condition. Moreover, due to the fact that the equations in the model~\eqref{EQ:SHG Robin One Way Coupled} are only one-way coupled, we are not limited to the usage of small boundary data $g$.

The measured interior data still take the form~\eqref{EQ:Data Form}. We will use data generated from $N_s\ge 1$ different boundary conditions $\{g_{j}\}_{j=1}^{N_s}$: $\{H_j\}_{j=1}^{N_s}$.



The numerical reconstructions are performed using standard least-squares optimization procedures that we will outline below. The computational implementation of the numerical simulations in this section can be found at \url{https://github.com/nsoedjak/Imaging-SHG}. All of the following numerical experiments can be reproduced by simply running the appropriate example file (e.g., \texttt{Experiment\_I\_gamma.m} for Numerical Experiment I).


\paragraph{Numerical Experiment I: reconstructing $\gamma$.} We start with reconstructing the coefficient $\gamma$, assuming all other coefficients are known. The reconstruction is achieved with an optimization algorithm that finds $\gamma$ by minimizing the functional
\begin{equation*}
	\Phi(\gamma) := \dfrac{1}{2}\sum_{j=1}^{N_s}\|\Gamma\sigma(|u_j|^2 + |v_j|^2) -H_j\|_{L^2(\Omega)}^2+\frac12\beta \|\nabla \gamma\|_{L^2(\Omega)}^2\,,
\end{equation*}
where we assume that we have collected data from $N_s$ different boundary conditions $\{g_j\}_{j=1}^{N_s}$. The regularization parameter $\beta$ will be selected with a trial and error approach. Following the standard adjoint-state method, we introduce the adjoint equations 
\begin{equation*}
	\begin{array}{rcll}
  	\Delta w_j+ (2k)^2 (1+\eta) w_j +i2k\sigma w_j &=& -\Big[\Gamma\sigma(|u_j|^2+|v_j|^2) - H_j\Big]\Gamma\sigma v_j^*,  &\mbox{in}\ \ \Omega\\
	w_j+i2k\bnu\cdot \nabla w_j &=& 0, & \mbox{on}\ \partial\Omega
	\end{array}
\end{equation*}
It is then straightforward to verify that the Fr\'echet derivative of $\Phi$ in direction $\delta \gamma$ can be written as:
\begin{align*}
    \Phi'(\gamma)[\delta\gamma] = \int_{\Omega} \left(2(2k)^2 \Re\sum_{j=1}^{N_s}  w_j u_j^2\right)\delta\gamma\,d\bx
    -\beta\int_{\Omega} (\Delta\gamma)\delta\gamma\,d\bx + \beta\int_{\partial\Omega} \pdr{\gamma}{\nu}\delta\gamma\,dS
\end{align*}
Once we have the gradient of the objective function with respect to $\gamma$, we feed it into a quasi-Newton optimization algorithm with the BFGS updating rule on the Hessian, implemented in {\rm MATLAB}. 
\begin{figure}[!htb]
\centering
\includegraphics[width=.31\linewidth]{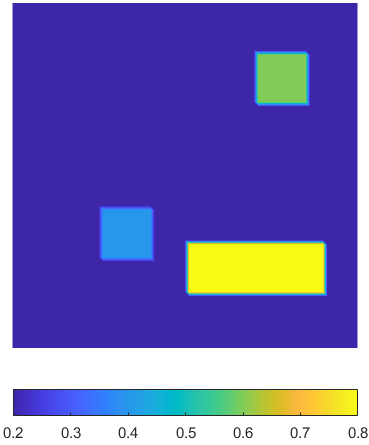}
\includegraphics[width=.31\linewidth]{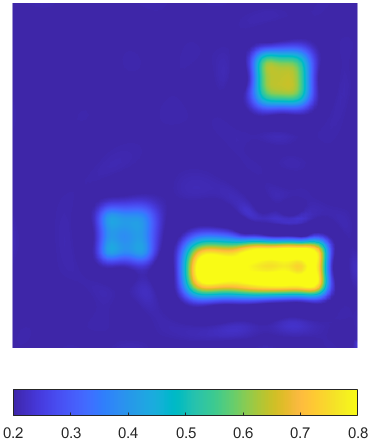}
\includegraphics[width=.31\linewidth]{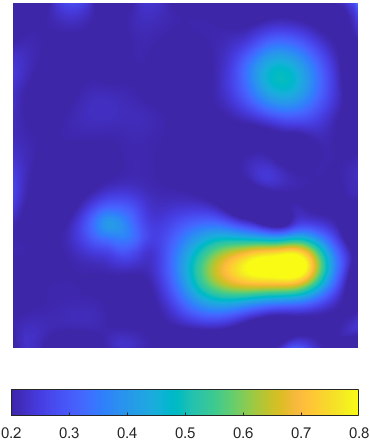}
\caption{True $\gamma$ (left), $\gamma$ reconstructed from noise-free data (middle), and $\gamma$ reconstructed from data containing 1\% random noise (right).}
\label{FIG:gamma}
\end{figure}

Figure~\ref{FIG:gamma} shows the reconstruction of a simple profile of $\gamma$ from both noise-free and noisy data. The regularization parameter is set to be $\beta=10^{-7}$ for this particular case. The quality of the reconstructions is reasonable by visual inspection. Similar levels of reconstruction quality are observed for various $\gamma$ profiles we tested. The regularization parameter is selected in a trial-and-error manner. The value of $\beta$ we used in the simulations may not be the ones to give the best reconstructions. However, we are not interested in tuning the regularization parameter to improve the reconstruction quality slightly. Therefore, we will not discuss this issue here.

\paragraph{Numerical Experiment II: reconstructing $(\eta, \sigma,\gamma)$.} In the second numerical example, we consider the case where $\Gamma$ is known but $\eta$, $\sigma$, and $\gamma$ are unknown. The inversions are done with a least-squares minimization algorithm that is similar to the one used in Numerical Experiment I. Figure~\ref{FIG:n-sigma-gamma noiseless data} shows that we are still able to obtain good numerical reconstructions, at least in the case when the profiles of $\eta$ and $\gamma$ are simple.
\begin{figure}[!htb]
\centering
\includegraphics[width=.31\linewidth]{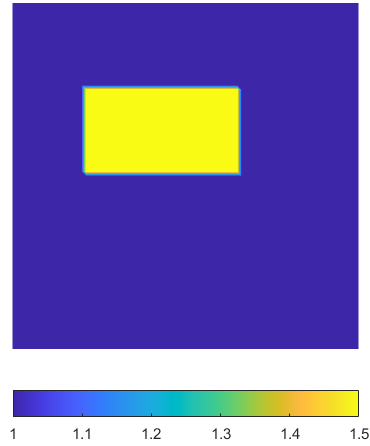}
\includegraphics[width=.31\linewidth]{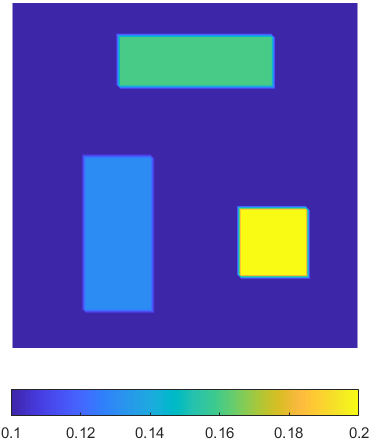}
\includegraphics[width=.31\linewidth]{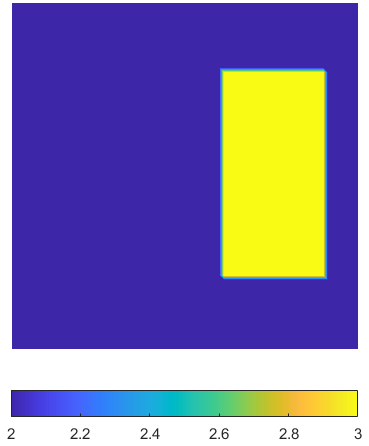}
\centering
\includegraphics[width=.31\linewidth]{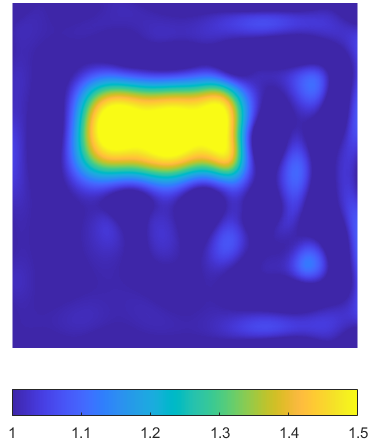}
\includegraphics[width=.31\linewidth]{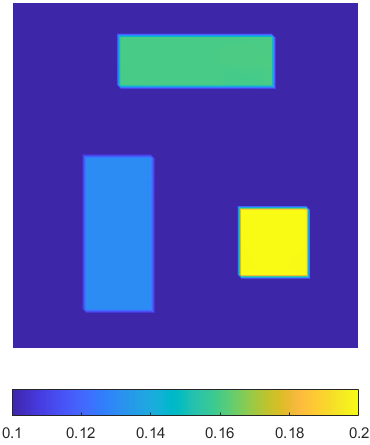}
\includegraphics[width=.31\linewidth]{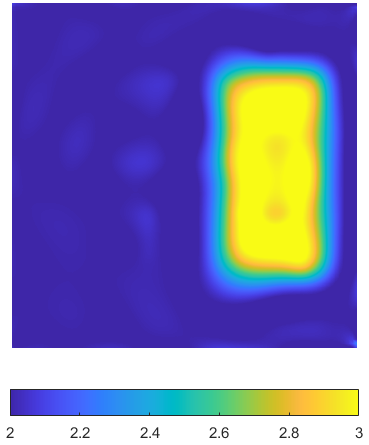}
\caption{True (top) and reconstructed (bottom) $\eta$ (left), $\sigma$ (middle), and $\gamma$ (right) in Numerical Experiment II.}
\label{FIG:n-sigma-gamma noiseless data}
\end{figure}

\paragraph{Numerical Experiment III: reconstructing $(\eta,\gamma,\Gamma)$.} In this example, we assume that $\sigma$ is known and we are interested in reconstructing $\eta$, $\gamma$ and $\Gamma$. Due to the fact that $\Gamma$ only appears in the measurement, not the PDE model, a naive least-squares minimization formulation like the ones in the previous examples will lead to unbalanced sensitivity between $\Gamma$ and the rest of the parameters. Hence we instead take a two-step reconstruction approach. In the first step, we use the ratio between measurements to eliminate $\Gamma$. That is, we minimize the functional
\begin{equation*}
	\Psi(\eta,\gamma) := \dfrac{1}{2}\sum_{j=2}^{N_s}\left\|\frac{|u_j|^2 + |v_j|^2}{|u_1|^2+|v_1|^2} - \frac{H_j}{H_1}\right\|_{L^2(\Omega)}^2 + \frac12\beta_1\|\nabla \eta\|_{L^2(\Omega)}^2 + \frac12\beta_2\|\nabla \gamma\|_{L^2(\Omega)}^2\,,
\end{equation*}
where we assume that we have collected data from $N_s$ different boundary conditions $\{g_j\}_{j=1}^{N_s}$. It is clear that $\Psi$ only depends on $\eta$ and $\gamma$, not $\Gamma$. The Fr\'echet derivatives of $\Psi$ can again be found using the standard adjoint-state method. For example, for the derivative with respect to $\gamma$, we introduce the adjoint equations 
\begin{equation*}
	\begin{array}{rcll}
  	\Delta w_j+ (2k)^2 (1+\eta) w_j +i2k\sigma w_j &=& -\left(\dfrac{|u_j|^2+|v_j|^2}{|u_1|^2+|v_1|^2}-\dfrac{H_j}{H_1}\right)\dfrac{1}{|u_1|^2+|v_1|^2}v_j^*,  &\mbox{in}\ \ \Omega\\
	w_j+i2k\bnu\cdot \nabla w_j &=& 0, & \mbox{on}\ \partial\Omega
	\end{array}
\end{equation*}
and 
\begin{equation*}
	\begin{array}{rcll}
  	\Delta z_j+ (2k)^2 (1+\eta) z_j +i2k\sigma z_j &=& \left(\dfrac{|u_j|^2+|v_j|^2}{|u_1|^2+|v_1|^2}-\dfrac{H_j}{H_1}\right)\dfrac{|u_j|^2+|v_j|^2}{(|u_1|^2+|v_1|^2)^2}v_1^*,  &\mbox{in}\ \ \Omega\\
	z_j+i2k\bnu\cdot \nabla z_j &=& 0, & \mbox{on}\ \partial\Omega.
	\end{array}
\end{equation*}
It is then straightforward to verify that the Fr\'echet derivative of $\Psi$ with respect to $\gamma$ in direction $\delta \gamma$ can be written as:
\begin{equation*}
    \Psi'_{\gamma}(\eta,\gamma)[\delta\gamma] = \int_{\Omega} \left(2(2k)^2\Re\sum_{j=1}^{N_s} \left[w_j u_j^2 + z_j u_1^2\right]\right)\delta\gamma\,d\bx
    -\beta_2\int_{\Omega} (\Delta\gamma)\delta\gamma\,d\bx + \beta_2\int_{\partial\Omega} \pdr{\gamma}{\nu}\delta\gamma\,dS\,.
\end{equation*}
The Fr\'echet derivative with respect to $\eta$ can be computed in a similar fashion. 
Once $\eta$ and $\gamma$ are reconstructed, we can reconstruct $\Gamma$ as
\begin{equation}\label{EQ:Gamma reconstruction}
    \Gamma = \frac{1}{N_s}\sum_{j=1}^{N_s} \frac{H_j}{\sigma (|u_j|^2+|v_j|^2)}\,.
\end{equation}

\begin{figure}[!htb]
\centering
\includegraphics[width=.31\linewidth]{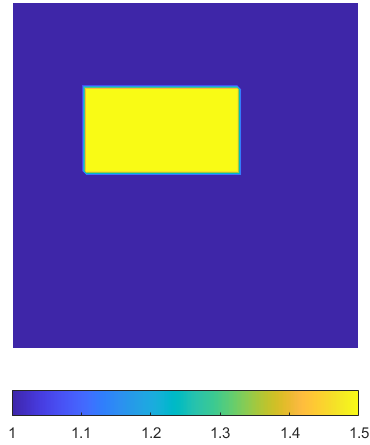}
\includegraphics[width=.31\linewidth]{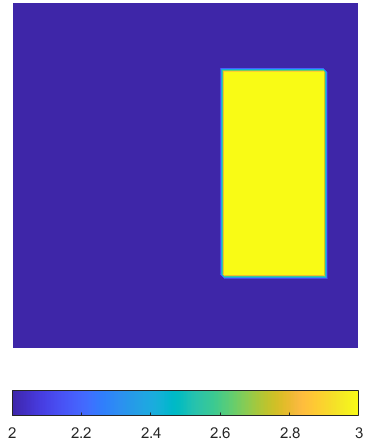}
\includegraphics[width=.31\linewidth]{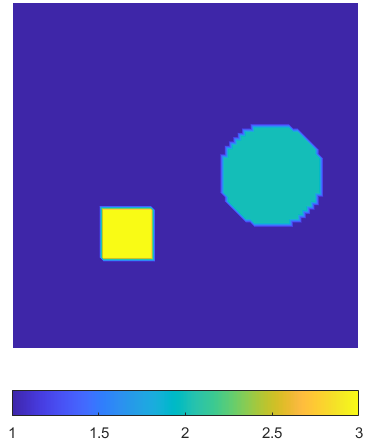}
\centering
\includegraphics[width=.31\linewidth]{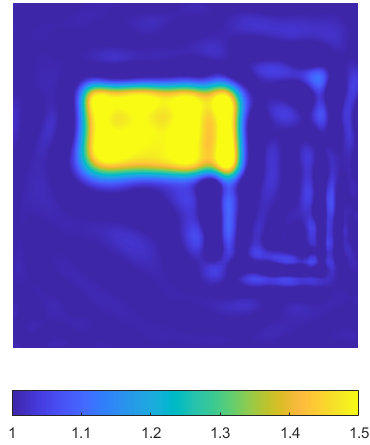}
\includegraphics[width=.31\linewidth]{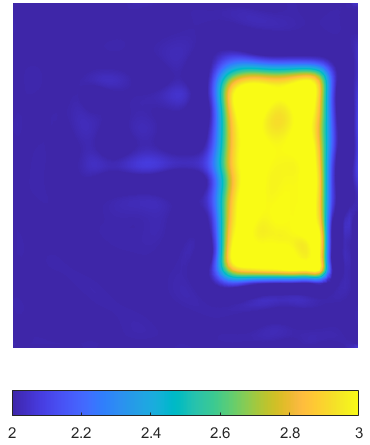}
\includegraphics[width=.31\linewidth]{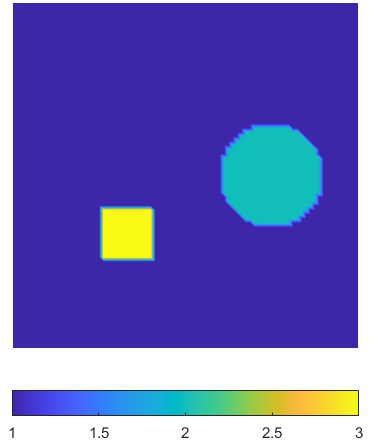}
\caption{True (top) and reconstructed (bottom) $\eta$ (left), $\gamma$ (middle), and $\Gamma$ (right).}
\label{FIG:n-gamma-Gamma noiseless data}
\end{figure}
A typical reconstruction is shown in Figure~\ref{FIG:n-gamma-Gamma noiseless data}. The reconstructions are highly accurate in this case.

\paragraph{Numerical Experiment IV: reconstructing $(\eta,\sigma,\gamma,\Gamma)$.}
Figure~\ref{FIG:n-sigma-gamma-Gamma noiseless data} shows a typical reconstruction of all four coefficients simultaneously. The reconstruction quality is high in the eyeball norm and can be characterized more precisely with numbers such as the relative $L^2$ error. Note from the reconstruction formula~\eqref{EQ:Gamma reconstruction} that any inaccuracies in the reconstruction of $\sigma$ will directly translate into artifacts in the reconstruction of $\Gamma$. This can be observed in Figure~\ref{FIG:n-sigma-gamma-Gamma noiseless data} (see columns 2 and 4), most notably near the edges of the square anomaly in $\sigma$.
\begin{figure}[!htb]
\centering
\includegraphics[width=.24\linewidth]{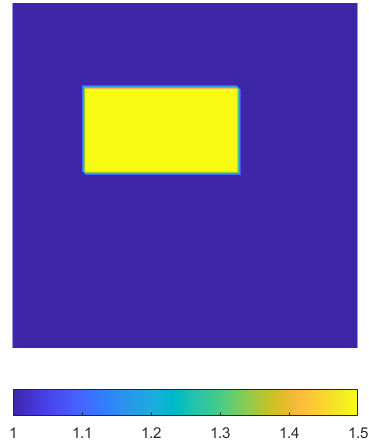}
\includegraphics[width=.24\linewidth]{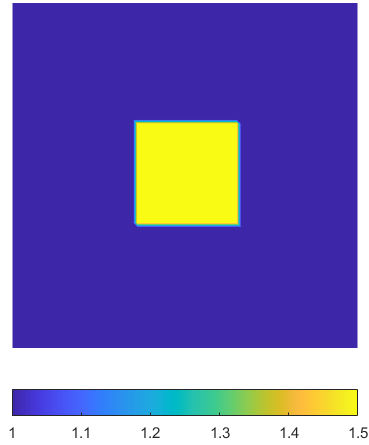}
\includegraphics[width=.24\linewidth]{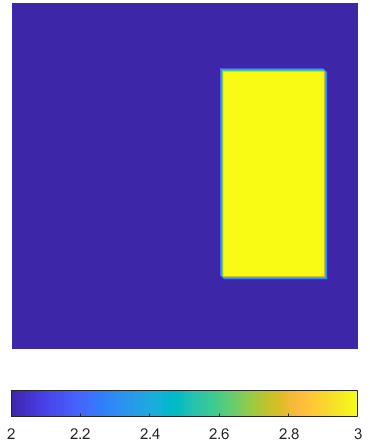}
\includegraphics[width=.24\linewidth]{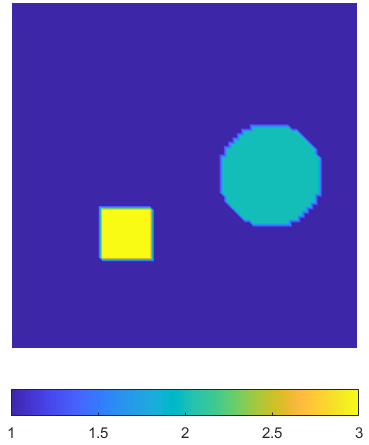}
\centering
\includegraphics[width=.24\linewidth]{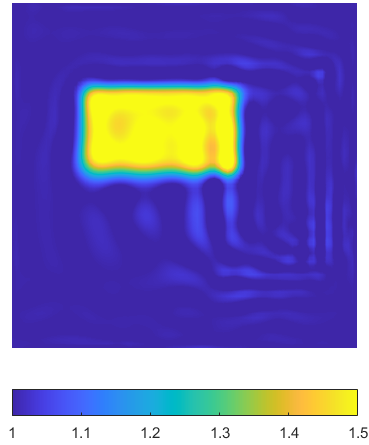}
\includegraphics[width=.24\linewidth]{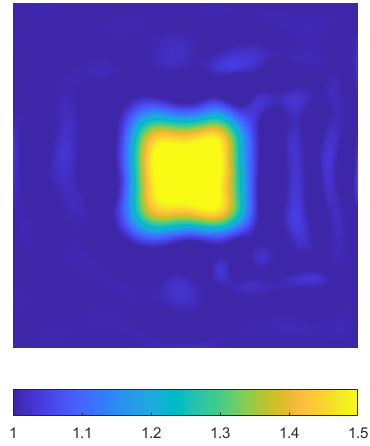}
\includegraphics[width=.24\linewidth]{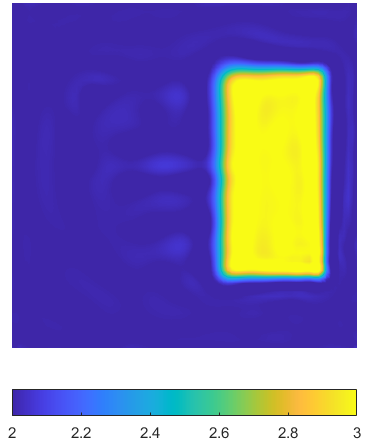}
\includegraphics[width=.24\linewidth]{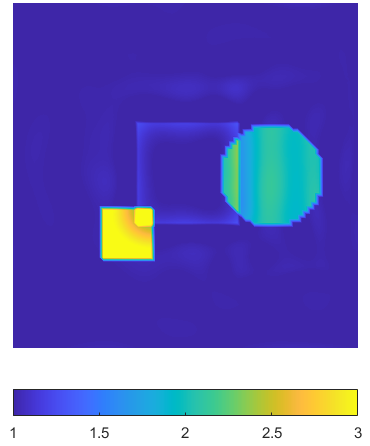}
\caption{True (top) and reconstructed (bottom) $(\eta, \sigma, \gamma, \Gamma)$. From left to right are $\eta$, $\sigma$, $\gamma$, and $\Gamma$.}
\label{FIG:n-sigma-gamma-Gamma noiseless data}
\end{figure}

\section{Concluding remarks}
\label{SEC:Concl}

We performed a systematic study on inverse problems to a system of coupled semilinear Helmholtz equations as the model for second harmonic generation in thermoacoustic imaging. We developed uniqueness and stability theory for the inverse problems utilizing the multilinearization technique. We showed, via both mathematical analysis and numerical simulations, that it is possible to reconstruct all four coefficients of interest from noisy interior data.

While our results show great promise for the solution of the inverse problems, several aspects of our study's technical side still need to be significantly improved. For instance, we have assumed the Dirichlet boundary condition for the generated second harmonic wave $v$ in model~\eqref{EQ:SHG}. This should certainly be replaced with homogeneous Robin-type boundary conditions that are more physical (as what we did in the computational experiments). Moreover, in Theorem~\ref{THM:Reconstruction of gamma}, we should be able to relax the requirement that the domain $\Omega$ is sufficiently small. In the same theorem, we should be able to remove the requirement on the additional Neumann boundary data to have a unique reconstruction of $\gamma$.

We have a few future directions in mind to continue the investigation from the perspective of practical applications. First, our mathematical results are mainly based on the assumption that the incident wave, that is, the Dirichlet boundary condition in system~\eqref{EQ:SHG}, is weak since this is the case where we can establish the well-posedness of the mathematical model. This assumption, however, severely limits the applicability of the analysis for practical applications as one needs to have a sufficiently strong boundary source to generate strong second-harmonic waves in order to see its impact on the data used for inversion.  Second, the linearization method requires access to a sequence of datasets generated from $\eps$-dependent boundary source. This is a large amount of data. It would be interesting to see if our uniqueness and stability results can be reproduced for a finite number of measurements. Third, it would be of great interest to see if one can perform a similar analysis on the same inverse problem to the Maxwell model of second-harmonic generation, such as the model introduced in~\cite{AsZh-JDE21}. In fact, the linearization machinery for the Maxwell model has already been built in~\cite{AsZh-JDE21}. However, it is not obvious whether or not our results can be generalized to the Maxwell model with the same type of data in a straightforward way.

\section*{Acknowledgments}

This work is partially supported by the National Science Foundation through grants DMS-1913309 and DMS-1937254. 

\appendix

\section{Well-posedness of system~\texorpdfstring{\eqref{EQ:SHG}}{}}
\label{APP:Well-posedness}

In this appendix, we establish the well-posedness of the boundary value problem \eqref{EQ:SHG} for sufficiently small boundary illuminations $g$ and $h$ using a standard contraction mapping theorem argument. 

We begin by recording a result on the well-posedness of the Helmholtz problem \eqref{EQ:Helmholtz Scalar}.
\begin{theorem}\label{THM:Helmholtz Well-Posedness} Let $\alpha\in (0,1)$, $q,f\in \cC^{0,\alpha}(\overline{\Omega}; \bbC)$, and $g\in \cC^{2,\alpha}(\partial\Omega; \bbC)$. If $\Im q >0$, then the boundary value problem
\begin{equation}\label{EQ:Helmholtz Scalar General}
	\begin{array}{rcll}
		\Delta u + qu &= &f, & \mbox{in}\ \Omega\\
		u & = & g, & \mbox{on}\ \partial\Omega
	\end{array}
\end{equation}
has a unique solution $u\in \cC^{2,\alpha}(\overline{\Omega};\bbC)$. Moreover, there exists a constant $C=C(\alpha,\Omega,q)$ such that the following Schauder estimate holds:
\begin{align}\label{EQ:Helmholtz Estimate}
    \|u\|_{\cC^{2,\alpha}(\overline{\Omega})} \le C(\|f\|_{\cC^{0,\alpha}(\overline{\Omega})} + \|g\|_{\cC^{2,\alpha}(\partial\Omega)})\,.
\end{align}
\end{theorem}

\begin{proof} The first task is to employ an energy method to show that \eqref{EQ:Helmholtz Scalar General} has at most one solution $u\in \cC^{2,\alpha}(\overline{\Omega};\bbC)$. To this end, suppose that $u\in \cC^{2,\alpha}(\overline{\Omega};\bbC)$ solves the homogeneous problem 
\begin{equation}
	\begin{array}{rcll}
		\Delta u + qu &= &0, & \mbox{in}\ \Omega\\
		u & = & 0, & \mbox{on}\ \partial\Omega\,.
	\end{array}
\end{equation}
Multiplying both sides of the PDE by $u^*$ and integrating over $\Omega$ results in 
\begin{align*}
\int_\Omega -|\nabla u|^2 + q|u|^2\, d\bx = 0\, ,
\end{align*}
whereupon taking imaginary parts yields $\int_\Omega \Im q |u|^2\, d\bx = 0$. The assumption $\Im q > 0$ then leads to $u\equiv 0$, as desired. This completes the proof of uniqueness.

Now that we have shown uniqueness, the existence of a solution $u\in \cC^{2,\alpha}(\overline{\Omega};\bbC)$ to the elliptic problem \eqref{EQ:Helmholtz Scalar General} follows from the Fredholm alternative: see \cite[Theorem 12.7]{AgDoNi-CPAM59} (which applies to elliptic operators with not only real-valued but also complex-valued coefficients).

Finally, from \cite[Theorem 7.3]{AgDoNi-CPAM59} we have the Schauder estimate 
\begin{align*}
    \|u\|_{\cC^{2,\alpha}(\overline{\Omega})} \le C(\|f\|_{\cC^{0,\alpha}(\overline{\Omega})} + \|g\|_{\cC^{2,\alpha}(\partial\Omega)}+\|u\|_{\cC^0(\overline{\Omega})})\,.
\end{align*}
On account of the problem having a \emph{unique} solution in $\cC^{2,\alpha}(\overline{\Omega};\bbC)$, the last term $\|u\|_{\cC^0(\overline{\Omega})}$ may be dropped (see Remark 2 following \cite[Theorem 7.3]{AgDoNi-CPAM59}) to arrive at the desired \eqref{EQ:Helmholtz Estimate}. 
\end{proof}

\begin{proof}[Proof of Theorem~\ref{THM:Well-Posedness}] The proof is a standard argument based on the Banach fixed point theorem. 

Before proceeding, we establish some notation. Let us define $q_1 := k^2(1+\eta) + ik\sigma$ and $q_2 := (2k)^2(1+\eta) + i2k\sigma$ for the sake of brevity of notation. If $X$ and $Y$ are two metric spaces, we shall equip the Cartesian product $X\times Y$ with any of the standard metrics, say the metric $d_{X\times Y}((x_1,y_1),(x_2,y_2)) := d_X(x_1,x_2) + d_Y(y_1,y_2)$. If $X$ and $Y$ are both complete, then so is $X\times Y$. Finally, for $\delta>0$ define the complete metric space
\begin{align*}
    \cX_\delta := \{f\in \cC^{2,\alpha}(\overline{\Omega}): \|f\|_{\cC^{2,\alpha}(\overline{\Omega})}\le \delta\}\,.
\end{align*}
To start, let $\eps>0$, $\delta>0$ and let $g,h\in \cC^{2,\alpha}(\partial\Omega;\bbC)$ with $\|g\|_{\cC^{2,\alpha}(\partial\Omega)}<\eps$ and $\|h\|_{\cC^{2,\alpha}(\partial\Omega)}<\eps$. We shall determine how small $\eps$ and $\delta$ need to be later.

In order to formulate the problem in terms of fixed points, define the operator $L^{-1}: \cC^{0,\alpha}(\overline{\Omega}) \times \cC^{0,\alpha}(\overline{\Omega}) \rightarrow \cC^{2,\alpha}(\overline{\Omega}) \times \cC^{2,\alpha}(\overline{\Omega})$ by setting $L^{-1}(f_1,f_2)$ as the unique solution $(U,V)\in \cC^{2,\alpha}(\overline{\Omega})\times \cC^{2,\alpha}(\overline{\Omega})$ to the problem 
\begin{equation*}
	\begin{array}{rcll}
  	\Delta U+ q_1U &=& f_1,
  	&\mbox{in}\ \ \Omega\\
  	\Delta V+ q_2V &=& f_2,
  	&\mbox{in}\ \ \Omega\\
        U = g,  && V = h, &\mbox{on}\ \ \partial\Omega\,.
	\end{array}
\end{equation*}
Then a solution $(u,v)\in \cX_\delta\times \cX_\delta$ to \eqref{EQ:SHG} is precisely the same as a fixed point in $\cX_\delta\times \cX_\delta$ of the operator $T$ defined by 
\begin{align*}
    T(\phi_1,\phi_2) := L^{-1}(-k^2\gamma \phi_1^* \phi_2,\ -(2k)^2\gamma \phi_1^2)\,.
\end{align*}
It remains to show that for sufficiently small $\epsilon>0$ and $\delta>0$,
\begin{enumerate}[label=(\roman*)]
    \item $T$ is a well-defined operator from $\cX_\delta\times \cX_\delta$ to itself, and
    \item $T$ is a contraction on $\cX_\delta\times \cX_\delta$.
\end{enumerate}
In order to perform the next several calculations, recall that $\cC^{0,\alpha}(\overline{\Omega})$ is a Banach algebra, meaning that 
\begin{align*}
    \|fg\|_{\cC^{0,\alpha}(\overline{\Omega})}\le \|f\|_{\cC^{0,\alpha}(\overline{\Omega})}\|g\|_{\cC^{0,\alpha}(\overline{\Omega})}
\end{align*}
for all $f,g\in \cC^{0,\alpha}(\overline{\Omega})$.

\paragraph{Proof of (i).} For all $(\phi_1,\phi_2)\in \cX_\delta\times \cX_\delta$, we compute 
\begin{align*}
    \|-k^2\gamma \phi_1^* \phi_2\|_{\cC^{0,\alpha}(\overline{\Omega})}
    \le C\|\phi_1\|_{\cC^{0,\alpha}(\overline{\Omega})}\|\phi_2\|_{\cC^{0,\alpha}(\overline{\Omega})} 
    \le C\|\phi_1\|_{\cC^{2,\alpha}(\overline{\Omega})}\|\phi_2\|_{\cC^{2,\alpha}(\overline{\Omega})} 
    \le C\delta^2
\end{align*}
and similarly 
\begin{align*}
    \|-(2k)^2\gamma \phi_1^2\|_{\cC^{0,\alpha}(\overline{\Omega})}
    \le C\|\phi_1\|_{\cC^{0,\alpha}(\overline{\Omega})} ^2 
    \le C\|\phi_1\|_{\cC^{2,\alpha}(\overline{\Omega})} ^2 
    \le C\delta^2
\end{align*}
Let $(U,V) := T(\phi_1,\phi_2)$. Combining the above estimates with the Schauder estimate \eqref{EQ:Helmholtz Estimate} for the Helmholtz equation then leads to
\begin{align*}
    \|U\|_{\cC^{2,\alpha}(\overline{\Omega})} 
    \le C(\|-k^2\gamma \phi_1^* \phi_2\|_{\cC^{0,\alpha}(\overline{\Omega})} + \|g\|_{\cC^{2,\alpha}(\partial\Omega)}) 
    \le C(\delta^2 + \eps)
\end{align*}
and 
\begin{align*}
    \|V\|_{\cC^{2,\alpha}(\overline{\Omega})} 
    \le C(\|-(2k)^2\gamma \phi_1^2\|_{\cC^{0,\alpha}(\overline{\Omega})} + \|h\|_{\cC^{2,\alpha}(\partial\Omega)}) 
    \le C(\delta^2 + \eps).
\end{align*}
We can force these quantities to be less than $\delta$ by choosing $\eps$ and $\delta$ sufficiently small. This implies that $T(\phi_1,\phi_2)\in \cX_\delta\times \cX_\delta$ as desired. 

\paragraph{Proof of (ii).} Let $(\phi_1,\phi_2), (\phi_1',\phi_2')\in \cX_\delta\times \cX_\delta$. Then we compute 
\begin{align*}
    \|-k^2\gamma \phi_1^*\phi_2 - (-k^2\gamma (\phi_1')^*\phi_2')\|_{\cC^{0,\alpha}(\overline{\Omega})} &\le C\|\phi_1^*\phi_2 - (\phi_1')^*\phi_2'\|_{\cC^{0,\alpha}(\overline{\Omega})}\\
    &= C\|\phi_1^*(\phi_2-\phi_2') + (\phi_1^*-(\phi_1')^*)\phi_2'\|_{\cC^{0,\alpha}(\overline{\Omega})}\\
    &\le C(\|\phi_1\|_{\cC^{0,\alpha}(\overline{\Omega})}\|\phi_2-\phi_2'\|_{\cC^{0,\alpha}(\overline{\Omega})} + \|\phi_1-\phi_1'\|_{\cC^{0,\alpha}(\overline{\Omega})}\|\phi_2'\|_{\cC^{0,\alpha}(\overline{\Omega})})\\
    &\le C(\|\phi_1\|_{\cC^{2,\alpha}(\overline{\Omega})}\|\phi_2-\phi_2'\|_{\cC^{2,\alpha}(\overline{\Omega})} + \|\phi_1-\phi_1'\|_{\cC^{2,\alpha}(\overline{\Omega})}\|\phi_2'\|_{\cC^{2,\alpha}(\overline{\Omega})})\\
    &\le C\delta \|(\phi_1,\phi_2) - (\phi_1',\phi_2')\|_{\cC^{2,\alpha}(\overline{\Omega})\times \cC^{2,\alpha}(\overline{\Omega})}
\end{align*}
and similarly 
\begin{align*}
    \|-(2k)^2\gamma\phi_1^2 - (-(2k)^2\gamma(\phi_1')^2)\|_{\cC^{0,\alpha}(\overline{\Omega})} &\le C\|\phi_1^2-(\phi_1')^2\|_{\cC^{0,\alpha}(\overline{\Omega})}\\
    &\le C\|\phi_1-\phi_1'\|_{\cC^{0,\alpha}(\overline{\Omega})}\|\phi_1+\phi_1'\|_{\cC^{0,\alpha}(\overline{\Omega})}\\
    &\le C\delta\|\phi_1-\phi_1'\|_{\cC^{2,\alpha}(\overline{\Omega})}\\
    &\le C\delta \|(\phi_1,\phi_2) - (\phi_1',\phi_2')\|_{\cC^{2,\alpha}(\overline{\Omega})\times \cC^{2,\alpha}(\overline{\Omega})}.
\end{align*}
Let $(U,V) := T(\phi_1,\phi_2)$ and $(U',V') := T(\phi_1',\phi_2')$. Then $U-U'$ and $V-V'$ satisfy 
\begin{equation*}
	\begin{array}{rcll}
  	\Delta (U-U')+ q_1(U-U') &=& -k^2\gamma \phi_1^*\phi_2 - (-k^2\gamma (\phi_1')^*\phi_2'),
  	&\mbox{in}\ \ \Omega\\
  	\Delta (V-V')+ q_2(V-V') &=& -(2k)^2\gamma\phi_1^2 - (-(2k)^2\gamma(\phi_1')^2),
  	&\mbox{in}\ \ \Omega\\
        U-U' = 0, \qquad V-V' &=& 0, &\mbox{on}\ \ \partial\Omega\,.
	\end{array}
\end{equation*}
Combining the above estimates with the Schauder estimate \eqref{EQ:Helmholtz Estimate} for the Helmholtz equation then leads to 
\begin{align*}
    \|U-U'\|_{\cC^{2,\alpha}(\overline{\Omega})} 
    \le C\|-k^2\gamma \phi_1^*\phi_2 - (-k^2\gamma (\phi_1')^*\phi_2')\|_{\cC^{0,\alpha}(\overline{\Omega})}
    \le C\delta \|(\phi_1,\phi_2) - (\phi_1',\phi_2')\|_{\cC^{2,\alpha}(\overline{\Omega})\times \cC^{2,\alpha}(\overline{\Omega})}
\end{align*}
and
\begin{align*}
    \|V-V'\|_{\cC^{2,\alpha}(\overline{\Omega})} 
    \le C\|-(2k)^2\gamma\phi_1^2 - (-(2k)^2\gamma(\phi_1')^2)\|_{\cC^{0,\alpha}(\overline{\Omega})}
    \le C\delta \|(\phi_1,\phi_2) - (\phi_1',\phi_2')\|_{\cC^{2,\alpha}(\overline{\Omega})\times \cC^{2,\alpha}(\overline{\Omega})}.
\end{align*}
We conclude that 
\begin{align*}
    \|T(\phi_1,\phi_2) - T(\phi_1',\phi_2')\|_{\cC^{2,\alpha}(\overline{\Omega})\times \cC^{2,\alpha}(\overline{\Omega})} &= \|U-U'\|_{\cC^{2,\alpha}(\overline{\Omega})} + \|V-V'\|_{\cC^{2,\alpha}(\overline{\Omega})}\\
    &\le C\delta \|(\phi_1,\phi_2) - (\phi_1',\phi_2')\|_{\cC^{2,\alpha}(\overline{\Omega})\times \cC^{2,\alpha}(\overline{\Omega})}.
\end{align*}
The factor $C\delta$ can be made strictly less than $1$ when $\delta$ is sufficiently small. This makes $T$ into a contraction, as desired. 

Having proved that $T$ is a contraction on the complete metric space $\cX_\delta\times \cX_\delta$, the Banach fixed point theorem guarantees that there exists a unique $(u,v)\in \cX_\delta\times \cX_\delta$ such that $T(u,v) = (u,v)$. As discussed earlier, this is equivalent to saying that there exists a unique $(u,v)\in \cX_\delta\times \cX_\delta$ satisfying the boundary value problem \eqref{EQ:SHG}. This completes the proof of the first part of the theorem.

\paragraph{Proof of the estimates \eqref{EQ:SHG Estimates}.} We perform a calculation similar to those in the proof of (i) to obtain 
\begin{align*}
    \|u\|_{\cC^{2,\alpha}(\overline{\Omega})} 
    &\le C(\|-k^2\gamma u^* v\|_{\cC^{0,\alpha}(\overline{\Omega})} + \|g\|_{\cC^{2,\alpha}(\partial\Omega)}) 
    \le C(\|u\|_{\cC^{0,\alpha}(\overline{\Omega})}\|v\|_{\cC^{0,\alpha}(\overline{\Omega})} + \|g\|_{\cC^{2,\alpha}(\partial\Omega)})\\
    &\le C(\|u\|_{\cC^{2,\alpha}(\overline{\Omega})}\delta + \|g\|_{\cC^{2,\alpha}(\partial\Omega)})
\end{align*}
When $\delta$ is sufficiently small, this implies that 
\begin{align*}
    \|u\|_{\cC^{2,\alpha}(\overline{\Omega})} \le C\|g\|_{\cC^{2,\alpha}(\partial\Omega)}\,,
\end{align*}
as desired. To get the estimate for $v$, we calculate 
\begin{align*}
    \|v\|_{\cC^{2,\alpha}(\overline{\Omega})} 
    &\le C(\|-(2k)^2 \gamma u^2\|_{\cC^{0,\alpha}(\overline{\Omega})} + \|h\|_{\cC^{2,\alpha}(\partial\Omega)}) 
    \le C(\|u\|_{\cC^{0,\alpha}(\partial\Omega)}^2 + \|h\|_{\cC^{2,\alpha}(\partial\Omega)})\\
    &\le C(\delta \|u\|_{\cC^{2,\alpha}(\overline{\Omega})} + \|h\|_{\cC^{2,\alpha}(\partial\Omega)}) 
    \le C(\delta\|g\|_{\cC^{2,\alpha}(\partial\Omega)} + \|h\|_{\cC^{2,\alpha}(\partial\Omega)})\\
    &\le C(\|g\|_{\cC^{2,\alpha}(\partial\Omega)} + \|h\|_{\cC^{2,\alpha}(\partial\Omega)})\,,
\end{align*}
as desired. The proof is complete.
\end{proof}

\section{Differentiability result for linearization}
\label{APP:Derivation of the linearization}

We provide here the mathematical justification of the linearization process we outlined in Section~\ref{SUBSEC:Linearization}. More precisely, we prove Theorem~\ref{THM:Linearization} by showing that $(u_\eps, v_\eps)$ and therefore $H_\eps$ are differentiable with respect to $\eps$.

\begin{proof}[Proof of Theorem~\ref{THM:Linearization}] Let us define $q_1 := k^2(1+\eta) + ik\sigma$ and $q_2 := (2k)^2(1+\eta) + i2k\sigma$ to ease notation.

To start the proof, let $\wt{u}^{(1)}$, $\wt{v}^{(1)}$, $\wt{u}^{(2)}$, $\wt{v}^{(2)}$ denote the unique functions in $\cC^{2,\alpha}(\overline{\Omega};\bbC)$ which satisfy equations \eqref{EQ:Order eps} and \eqref{EQ:Order eps2}. That is, 
\begin{equation*}
	\begin{array}{rcll}
  	\Delta \wt{u}^{(1)}+ q_1\wt{u}^{(1)} &=& 0,
  	&\mbox{in}\ \ \Omega\\
  	\Delta \wt{v}^{(1)}+ q_2\wt{v}^{(1)} &=& 0,
  	&\mbox{in}\ \ \Omega\\
	    \wt{u}^{(1)} = g_1,  && \wt{v}^{(1)} = h_1, &\mbox{on}\ \ \partial\Omega
	\end{array}
\end{equation*}
and 
\begin{equation*}
	\begin{array}{rcll}
  	\Delta \wt{u}^{(2)}+ q_1\wt{u}^{(2)} &=& -2k^2\gamma \wt{u}^{(1)*}\wt{v}^{(1)},
  	&\mbox{in}\ \ \Omega\\
  	\Delta \wt{v}^{(2)}+ q_2\wt{v}^{(2)} &=& -2(2k)^2\gamma (\wt{u}^{(1)})^2,
  	&\mbox{in}\ \ \Omega\\
        \wt{u}^{(2)}= g_2, && \wt{v}^{(2)} = h_2,&\mbox{on}\ \ \partial\Omega\,.
	\end{array}
\end{equation*}
(Note that the existence and uniqueness of these functions is guaranteed by Theorem \ref{THM:Helmholtz Well-Posedness}.) Define now the ``remainder" terms 
\begin{equation}\label{EQ:Remainders Definition}
\begin{aligned}
    \mu_\eps &:= u_\eps - \eps\wt{u}^{(1)} - \frac12\eps^2\wt{u}^{(2)},\\
    \nu_\eps &:= v_\eps - \eps\wt{v}^{(1)} - \frac12\eps^2\wt{v}^{(2)}.
\end{aligned}
\end{equation}
We wish to show that $\mu_\eps$ and $\nu_\eps$ are in a certain sense ``$o(\eps^2)$" as $\eps\rightarrow 0$. This will be accomplished in two rounds of estimates on $\mu_\eps$, $\nu_\eps$, $u_\eps$ and $v_\eps$. 

\paragraph{Round 1 estimates.} We begin by using the linearity of the operators $\Delta + q_1$ and $\Delta + q_2$ to find that 
\begin{equation}\label{EQ:Remainders}
	\begin{array}{rcll}
  	\Delta \mu_\eps + q_1\mu_\eps &=& -k^2\gamma [u_\eps^*v_\eps - \eps^2 \wt{u}^{(1)*}\wt{v}^{(1)}],
  	&\mbox{in}\ \ \Omega\\
  	\Delta \nu_\eps + q_2\nu_\eps &=& -(2k)^2\gamma [u_\eps^2 - \eps^2(\wt{u}^{(1)})^2],
  	&\mbox{in}\ \ \Omega\\
   \mu_\eps = 0, && \nu_\eps = 0,&\mbox{on}\ \ \partial\Omega\,.
	\end{array}
\end{equation}
To obtain control on the size of the right hand sides, we utilize the well-posedness result Theorem \ref{THM:Well-Posedness} to see that 
\begin{align*}
    \|u_\eps\|_{\cC^{0,\alpha}(\overline{\Omega})}\le \|u_\eps\|_{\cC^{2,\alpha}(\overline{\Omega})} \le C\left(\left\|\eps g_1 + \frac12\eps^2 g_2\right\|_{\cC^{2,\alpha}(\partial\Omega)} + \left\|\eps h_1 + \frac12\eps^2 h_2\right\|_{\cC^{2,\alpha}(\partial\Omega)}\right) \le C \eps,
\end{align*}
and similarly for $v_\eps$. Here, $C=C(\alpha,\Omega,\eta,\sigma,\gamma,g_1,g_2)$ is a constant not depending on $\eps$. We can write these bounds succinctly as 
\begin{align}\label{EQ:Solution Bounds}
    u_\eps = \cO_{\cC^{0,\alpha}(\overline{\Omega})}(\eps),\qquad v_\eps = \cO_{\cC^{0,\alpha}(\overline{\Omega})}(\eps).
\end{align}
In order to perform the next several calculations, recall that $\cC^{0,\alpha}(\overline{\Omega})$ is a Banach algebra, meaning that 
\begin{align*}
    \|fg\|_{\cC^{0,\alpha}(\overline{\Omega})}\le \|f\|_{\cC^{0,\alpha}(\overline{\Omega})}\|g\|_{\cC^{0,\alpha}(\overline{\Omega})}
\end{align*}
for all $f,g\in \cC^{0,\alpha}(\overline{\Omega})$. With the help of this property, we plug the bounds \eqref{EQ:Solution Bounds} into the right hand sides of \eqref{EQ:Remainders} to discover that
\begin{align}\label{EQ:RHS Bounds}
    -k^2\gamma [u_\eps^*v_\eps - \eps^2 \wt{u}^{(1)*}\wt{v}^{(1)}] = \cO_{\cC^{0,\alpha}(\overline{\Omega})}(\eps^2),\qquad -(2k)^2\gamma [u_\eps^2 - \eps^2(\wt{u}^{(1)})^2] = \cO_{\cC^{0,\alpha}(\overline{\Omega})}(\eps^2).
\end{align}
The Schauder estimate \eqref{EQ:Helmholtz Estimate} for the Helmholtz equation applied to \eqref{EQ:Remainders} then gives 
\begin{align}
    \mu_\eps = \cO_{\cC^{2,\alpha}(\overline{\Omega})}(\eps^2),\qquad \nu_\eps = \cO_{\cC^{2,\alpha}(\overline{\Omega})}(\eps^2),
\end{align}
and in particular
\begin{align}\label{EQ:Remainders Bounds}
    \mu_\eps = \cO_{\cC^{0,\alpha}(\overline{\Omega})}(\eps^2),\qquad \nu_\eps = \cO_{\cC^{0,\alpha}(\overline{\Omega})}(\eps^2).
\end{align}

\paragraph{Round 2 estimates.} Using the estimates from Round 1 and recalling the definition \eqref{EQ:Remainders Definition} of the remainder terms $\mu_\eps$ and $\nu_\eps$, we can now refine the bounds \eqref{EQ:RHS Bounds} on the right hand sides of \eqref{EQ:Remainders}:
\begin{align*}
    &-k^2\gamma [u_\eps^*v_\eps - \eps^2 \wt{u}^{(1)*}\wt{v}^{(1)}]\\
    &= -k^2\gamma\left[\left(\eps\wt{u}^{(1)*} + \frac12\eps^2\wt{u}^{(2)*}+\cO_{\cC^{0,\alpha}(\overline{\Omega})}(\eps^2)\right)\left(\eps \wt{v}^{(1)} + \frac12\eps^2\wt{v}^{(2)}+\cO_{\cC^{0,\alpha}(\overline{\Omega})}(\eps^2)\right) - \eps^2 \wt{u}^{(1)*}\wt{v}^{(1)}\right]\\
    &= \cO_{\cC^{0,\alpha}(\overline{\Omega})}(\eps^3)   
\end{align*}
and 
\begin{align*}
    -(2k)^2\gamma [u_\eps^2 - \eps^2(\wt{u}^{(1)})^2] &= -(2k)^2\gamma \left[\left(\eps\wt{u}^{(1)} + \frac12\eps^2\wt{u}^{(2)} + \cO_{\cC^{0,\alpha}(\overline{\Omega})}(\eps^2)\right)^2 - \eps^2(\wt{u}^{(1)})^2\right] = \cO_{\cC^{0,\alpha}(\overline{\Omega})}(\eps^3).
\end{align*}
With these improved bounds in hand, we again apply the Schauder estimate \eqref{EQ:Helmholtz Estimate} to \eqref{EQ:Remainders} to obtain the following bounds for the remainder terms $\mu_\eps$ and $\nu_\eps$:
\begin{align*}
    \mu_\eps = \cO_{\cC^{0,\alpha}(\overline{\Omega})}(\eps^3),\qquad \nu_\eps = \cO_{\cC^{0,\alpha}(\overline{\Omega})}(\eps^3).
\end{align*}
Note that this is a refinement over the previous remainder bounds \eqref{EQ:Remainders Bounds}. This concludes the Round 2 estimates.

Therefore from the definition \eqref{EQ:Remainders Definition} of $\mu_\eps$ and $\nu_\eps$ we conclude that 
\begin{align*}
    u_\eps &= \eps\wt{u}^{(1)} + \frac12\eps^2\wt{u}^{(2)} + \cO_{\cC^{0,\alpha}(\overline{\Omega})}(\eps^3),\\
    v_\eps &= \eps\wt{v}^{(1)} + \frac12\eps^2\wt{v}^{(2)} + \cO_{\cC^{0,\alpha}(\overline{\Omega})}(\eps^3),
\end{align*}
and in particular for each $\bx\in\overline{\Omega}$ we have the pointwise estimates 
\begin{align*}
    u_\eps(\bx) &= \eps\wt{u}^{(1)}(\bx) + \frac12\eps^2\wt{u}^{(2)}(\bx) + \cO(\eps^3),\\
    v_\eps(\bx) &= \eps\wt{v}^{(1)}(\bx) + \frac12\eps^2\wt{v}^{(2)}(\bx) + \cO(\eps^3),
\end{align*}
as $\eps\rightarrow 0$.

\paragraph{Asymptotic expansion of $H_{\eps}$.} To finish the proof, for each $\bx\in\overline{\Omega}$ we compute:
\begin{align*}
    H_\eps(\bx) &= \Gamma(\bx)\sigma(\bx)[u_\eps(\bx)u^*_\eps(\bx) + v_\eps(\bx)v^*_\eps(\bx)]\\
    &=\Gamma(\bx)\sigma(\bx)\Bigg[\left(\eps\wt{u}^{(1)}(\bx) + \frac12\eps^2\wt{u}^{(2)}(\bx) + \cO(\eps^3)\right)\left(\eps\wt{u}^{(1)*}(\bx) + \frac12\eps^2\wt{u}^{(2)*}(\bx) + \cO(\eps^3)\right)\\
    &\qquad + \left(\eps\wt{v}^{(1)}(\bx) + \frac12\eps^2\wt{v}^{(2)}(\bx) + \cO(\eps^3)\right)\left(\eps\wt{v}^{(1)*}(\bx) + \frac12\eps^2\wt{v}^{(2)*}(\bx) + \cO(\eps^3)\right) \Bigg]\\
    &= \eps H^{(1)}(\bx) + \frac12\eps^2 H^{(2)}(\bx) + \frac16\eps^3 H^{(3)}(\bx) + \cO(\eps^4),
\end{align*}
where 
\begin{align*}
    H^{(1)} &:= 0,\\
    H^{(2)} &:=2\Gamma \sigma (\wt{u}^{(1)*} \wt{u}^{(1)} + \wt{v}^{(1)*} \wt{v}^{(1)}),\\
    H^{(3)} &:= 3\Gamma\sigma\left(\wt{u}^{(1)*}\wt{u}^{(2)}+\wt{u}^{(1)}\wt{u}^{(2)*}+\wt{v}^{(1)*}\wt{v}^{(2)}+\wt{v}^{(1)}\wt{v}^{(2)*}\right).
\end{align*}
These are precisely the equations \eqref{EQ:Data Order eps}, \eqref{EQ:Data Order eps2}, and \eqref{EQ:Data Order eps3}, respectively. This completes the proof.
\end{proof}

\section{Elliptic system theory}
\label{APP:Elliptic system theory}
For the sake of completeness, in this appendix we review the elliptic system theory that appears in the proof of Theorem~\ref{THM:Reconstruction of gamma}. We mostly follow the presentation of~\cite{Bal-CM13}.

Let $M\ge N$ be positive integers and consider the following system of $M$ partial differential equations in $N$ unknown functions $v_1,\dots, v_N$ defined on an open set $\Omega$:
\begin{align*}
    \cA(\bx,D) v = \cS\quad\mbox{in}\ \Omega.
\end{align*}
Here $v=(v_1,\dots, v_N)$, $D=(\partial_{x_1},\dots,\partial_{x_n})$, and $\cA(\bx,D)$ is an $M\times N$ matrix linear partial differential operator. That is, for each $1\le i\le M$ and $1\le j\le N$, the entry $\cA_{ij}(\bx,D)$ is a linear partial differential operator, and the above matrix equation means that 
\begin{align*}
    \sum_{j=1}^N \cA_{ij}(\bx,D)v_j = \cS_i,\qquad 1\le i\le M.
\end{align*}
To each row $1\le i\le M$ let us now associate an integer $s_i$, and to each column $1\le j\le N$ let us associate an integer $t_j$. Let us choose the numbers in such a way that the order of the partial differential operator $\cA_{ij}(\bx,D)$ is no greater than $s_i+t_j$. (If $s_i+t_j<0$, then we require that $\cA_{ij}(\bx,D)=0$.)

The \emph{principal part} $\cA_0(\bx,D)$ of $\cA(\bx,D)$ is defined as the $M\times N$ matrix linear partial differential operator such that $(\cA_0)_{ij}(\bx,D)$ consists of the terms in $\cA_{ij}(\bx,D)$ of order exactly $s_i+t_j$.

\begin{definition} The matrix partial differential operator $\cA$ is called \emph{elliptic} if such Douglis-Nirenberg numbers $(s_i)_{1\le i\le M}$ and $(t_j)_{1\le j\le N}$ exist, and the matrix $\cA_0(\bx,\bxi)$ has full rank $N$ for each $\bx\in\Omega$ and $\bxi\in \bbS^{n-1}$. (The matrix $\cA_0(\bx,\bxi)$ is called the \emph{symbol} of the operator $\cA_0$.)
\end{definition}

We now summarize the parts of \cite[Section 3]{Bal-CM13} that are relevant for the proof of Theorem~\ref{THM:Reconstruction of gamma}. Assume from now on that $\cA$ is an elliptic operator with continuous coefficients and Douglis-Nirenberg numbers $(s_i)_{1\le i\le M}$ and $(t_j)_{1\le j\le N}$ such that 
\begin{equation}\label{EQ:Douglis-Nirenberg Numbers Condition}
\begin{array}{rcll}
    s_i &=& 0, &1\le i\le M,\\
    t_j &=& \tau, &1\le j\le N,
\end{array}
\end{equation}
for some positive integer $\tau$. In the following, we will consider the Dirichlet boundary value problem 
\begin{equation}\label{EQ:Dirichlet Problem}
\begin{array}{rcll}
    \cA(\bx,D)v &=& \cS &\mbox{in}\ \Omega,\\
    \left(\pdr{}{\nu}\right)^q v_j &=& \phi_{qj} &\mbox{on}\ \partial\Omega,\quad 0\le q\le \tau-1, \quad 1\le j\le N.
\end{array}
\end{equation}
For this boundary value problem, we have the following estimate.

\begin{theorem}{\cite[p. 13]{Bal-CM13}}\label{THM:Elliptic Dirichlet Stability Estimate} Let $p>1$ and $\ell \ge 0$. Assume that $\cS_i\in W^{\ell,p}(\Omega)$ for all $1\le i\le M$ and $\phi_{qj}\in W^{\ell+\tau-q-\frac{1}{p},p}(\partial\Omega)$ for all $0\le q\le \tau-1$ and $1\le j\le N$. Also, assume that all the coefficients in $\cA$ are in $\cC^{\ell}(\overline{\Omega})$ \cite[Theorem 1.1]{Solonnikov-JSM73}. Then the following elliptic regularity  estimate holds for the boundary value problem \eqref{EQ:Dirichlet Problem}:
\begin{align}\label{EQ:Elliptic Dirichlet Stability Estimate}
    \sum_{j=1}^{N} \|v_j\|_{W^{\ell+\tau,p}(\Omega)} &\le C\left(\sum_{i=1}^M \|\cS_i\|_{W^{\ell,p}(\Omega)} + \sum_{q,j}\|\phi_{qj}\|_{W^{\ell+\tau-q-\frac{1}{p},p}(\partial\Omega)}\right) + C_2\sum_{j=1}^N \|v_j\|_{L^p(\Omega)}
\end{align}
\end{theorem}


The following result on uniqueness of solutions in sufficiently small domains is used in the proof of Theorem~\ref{THM:Reconstruction of gamma}.
\begin{theorem}{\cite[Theorem 3.5]{Bal-CM13}}\label{THM:Small Domain Uniqueness} Let $\bx_0\in \Omega$. Then there exists $\epsilon>0$ such that for every small domain $\Omega'\subset B(\bx_0,\epsilon)$, the only solution to the homogeneous Dirichlet boundary value problem
\begin{equation}
\begin{array}{rcll}
    \cA(\bx,D)v &=& 0 &\mbox{in}\ \Omega',\\
    \left(\pdr{}{\nu}\right)^q v_j &=& 0 &\mbox{on}\ \partial\Omega',\quad 0\le q\le \tau-1, \quad 1\le j\le N.
\end{array}
\end{equation}
in $\Omega'$ is the trivial solution $v=0$. In particular, \eqref{EQ:Elliptic Dirichlet Stability Estimate} holds with $C_2=0$.
\end{theorem}

{\small

\begin{thebibliography}{10}

\bibitem{AgDoNi-CPAM59}
{\sc S.~Agmon, A.~Douglis, and L.~Nirenberg}, {\em Estimates near the boundary
  for solutions of elliptic partial differential equations satisfying general
  boundary conditions. {I}}, Comm. Pure Appl. Math., 12 (1959), pp.~623--727.

\bibitem{AkBeDaElLiMi-JIIP17}
{\sc H.~Akhouayri, M.~Bergounioux, A.~{Da Silva}, P.~Elbau, A.~Litman, and
  L.~Mindrinos}, {\em Quantitative thermoacoustic tomography with microwaves
  sources}, J. Inverse Ill-Posed Probl., 25 (2017), pp.~703--717.

\bibitem{JeEl-IP20}
{\sc H.~Al~Jebawy and A.~El~Badia}, {\em Direct algorithm for reconstructing
  small absorbers in thermoacoustic tomography problem from a single data},
  Inverse Problems, 36 (2020), p.~065010.

\bibitem{Alberti-arXiv22}
{\sc G.~S. Alberti}, {\em Non-zero constraints in elliptic {PDE} with random
  boundary values and applications to hybrid inverse problems},
  arXiv:2205.00994,  (2022).

\bibitem{Ambrosio-IM04}
{\sc L.~Ambrosio}, {\em Transport equation and {Cauchy} problem for {BV} vector
  fields}, Inventiones Mathematicae, 158 (2004), p.~227.

\bibitem{AmGaJiNg-ARMA12}
{\sc H.~Ammari, J.~Garnier, W.~Jing, and L.~Nguyen}, {\em Quantitative
  thermo-acoustic imaging: An exact reconstruction formula}, Submitted to
  Archive for Rational Mechanics and Analysis,  (2011).

\bibitem{AsZh-JDE21}
{\sc Y.~M. Assylbekov and T.~Zhou}, {\em Inverse problems for nonlinear
  maxwell's equations with second harmonic generation}, Journal of Differential
  Equations, 296 (2021), pp.~148--169.

\bibitem{Bal-CM13}
{\sc G.~Bal}, {\em Hybrid inverse problems and redundant systems of partial
  differential equations}, in Inverse Problems and Applications, P.~Stefanov,
  A.~Vasy, and M.~Zworski, eds., vol.~615 of Contemporary Mathematics, American
  Mathematical Society, 2013, pp.~15--48.

\bibitem{BaRe-IP11}
{\sc G.~Bal and K.~Ren}, {\em Multi-source quantitative {PAT} in diffusive
  regime}, Inverse Problems, 27 (2011).
\newblock 075003.

\bibitem{BaRe-IP12}
\leavevmode\vrule height 2pt depth -1.6pt width 23pt, {\em On multi-spectral
  quantitative photoacoustic tomography in diffusive regime}, Inverse Problems,
  28 (2012).
\newblock 025010.

\bibitem{BaReUhZh-IP11}
{\sc G.~Bal, K.~Ren, G.~Uhlmann, and T.~Zhou}, {\em Quantitative
  thermo-acoustics and related problems}, Inverse Problems, 27 (2011).
\newblock 055007.

\bibitem{BaZh-IP14}
{\sc G.~Bal and T.~Zhou}, {\em Hybrid inverse problems for a system of
  {Maxwell’s} equations}, Inverse Problems, 30 (2014).
\newblock 055013.

\bibitem{BeBrPr-IP19}
{\sc M.~Bergounioux, {\'E}.~Bretin, and Y.~Privat}, {\em How to position
  sensors in thermo-acoustic tomography}, Inverse Problems, 35 (2019),
  p.~074003.

\bibitem{BoLiMaSc-IP17}
{\sc L.~Borcea, W.~Li, A.~Mamonov, and J.~C. Schotland}, {\em Second-harmonic
  imaging in random media}, Inverse Problems, 33 (2017).
\newblock 065004.

\bibitem{BoWo-Book99}
{\sc M.~Born and E.~Wolf}, {\em {Principles} of {Optics}: {Electromagnetic}
  {Theory} of {Propagation}, {Interference} and {Diffraction} of {Light}},
  Cambridge University Press, New York, 1999.

\bibitem{BoCr-SIAM06}
{\sc F.~Bouchut and G.~Crippa}, {\em Uniqueness, renormalization and smooth
  approximations for linear transport equations}, SIAM J. Math. Anal., 38
  (2006), pp.~1316--1328.

\bibitem{Boyd-Book20}
{\sc R.~W. Boyd}, {\em Nonlinear optics}, Academic press, 2020.

\bibitem{Choulli-arXiv2022}
{\sc M.~Choulli}, {\em Stable determination of the nonlinear term in a
  quasilinear elliptic equation by boundary measurements}, arXiv:2205.16000,
  (2022).

\bibitem{CoLe-DMJ02}
{\sc F.~Colombini and N.~Lerner}, {\em Uniqueness of continuous solutions for
  {BV} vector fields}, Duke Math. J., 111 (2002), pp.~357--384.

\bibitem{CrLiSh-M2AS23}
{\sc M.~Cristofol, S.~Li, and Y.~Shang}, {\em Carleman estimates and some
  inverse problems for the coupled quantitative thermoacoustic equations by
  partial boundary layer data. part ii: Some inverse problems}, Mathematical
  Methods in the Applied Sciences,  (2023).

\bibitem{DiLi-AM89}
{\sc R.~J. DiPerna and P.-L. Lions}, {\em On the {Cauchy} problem for
  {Boltzmann} equations: global existence and weak stability}, Ann. Math., 130
  (1989), pp.~321--366.

\bibitem{FeLiLi-arXiv21}
{\sc A.~Feizmohammadi, T.~Liimatainen, and Y.-H. Lin}, {\em An inverse problem
  for a semilinear elliptic equation on conformally transversally anisotropic
  manifolds}, arXiv:2112.08305,  (2021).

\bibitem{FrTeGaBlNeMa-JOSA15}
{\sc J.~Franc\'{e}s, J.~Tervo, S.~Gallego, S.~Bleda, C.~Neipp, and
  A.~M\'{a}rquez}, {\em Split-field finite-difference time-domain method for
  second-harmonic generation in two-dimensionally periodic structures}, J. Opt.
  Soc. Am. B, 32 (2015), pp.~664--669.

\bibitem{HaLi-NA23}
{\sc B.~Harrach and Y.-H. Lin}, {\em Simultaneous recovery of piecewise
  analytic coefficients in a semilinear elliptic equation}, Nonlinear Analysis,
  228 (2023), p.~113188.

\bibitem{Isakov-ARMA93}
{\sc V.~Isakov}, {\em On uniqueness in inverse problems for semilinear
  parabolic equations}, Arch. Rational Mech. Anal., 124 (1993), pp.~1--12.

\bibitem{Isakov-Book06}
\leavevmode\vrule height 2pt depth -1.6pt width 23pt, {\em Inverse {Problems}
  for {Partial} {Differential} {Equations}}, Springer-Verlag, New York,
  second~ed., 2006.

\bibitem{Kian-Nonlinearity23}
{\sc Y.~Kian}, {\em Lipschitz and h{\"o}lder stable determination of nonlinear
  terms for elliptic equations}, Nonlinearity, 36 (2023), p.~1302.

\bibitem{KrUh-arXiv19}
{\sc K.~Krupchyk and G.~Uhlmann}, {\em Partial data inverse problems for
  semilinear elliptic equations with gradient nonlinearities},
  arXiv:1909.08122v1,  (2019).

\bibitem{KrUh-PAMS20}
\leavevmode\vrule height 2pt depth -1.6pt width 23pt, {\em A remark on partial
  data inverse problems for semilinear elliptic equations}, Proceedings of the
  AMS,  (2019).

\bibitem{LaLi-NA22}
{\sc R.-Y. Lai and Y.-H. Lin}, {\em Inverse problems for fractional semilinear
  elliptic equations}, Nonlinear Analysis, 216 (2022), p.~112699.

\bibitem{LaReZh-SIAM22}
{\sc R.-Y. Lai, K.~Ren, and T.~Zhou}, {\em Inverse transport and diffusion
  problems in photoacoustic imaging with nonlinear absorption}, SIAM J. Appl.
  Math., 82 (2022), pp.~602--624.
\newblock arXiv:2107.08118.

\bibitem{LaLiLiSa-JMPA21}
{\sc M.~Lassas, T.~Liimatainen, Y.-H. Lin, and M.~Salo}, {\em Inverse problems
  for elliptic equations with power type nonlinearities}, Journal de
  math{\'e}matiques pures et appliqu{\'e}es, 145 (2021), pp.~44--82.

\bibitem{LuZh-arXiv23}
{\sc S.~Lu and J.~Zhai}, {\em Increasing stability of a linearized inverse
  boundary value problem for a nonlinear schr\"{o}dinger equation on
  transversally anisotropic manifolds}, arXiv:2301.07875,  (2023).

\bibitem{Solonnikov-JSM73}
{\sc V.~A. Solonnikov}, {\em Overdetermined elliptic boundary-value problems},
  J. Sov. Math., 1 (1973), pp.~477--512.

\bibitem{SzKi-JOSA18}
{\sc T.~Szarvas and Z.~Kis}, {\em Numerical simulation of nonlinear second
  harmonic wave generation by the finite difference frequency domain method},
  J. Opt. Soc. Am. B, 35 (2018), pp.~731--740.

\bibitem{UhZh-JMPA21}
{\sc G.~Uhlmann and J.~Zhai}, {\em On an inverse boundary value problem for a
  nonlinear elastic wave equation}, Journal de Math{\'e}matiques Pures et
  Appliqu{\'e}es, 153 (2021), pp.~114--136.

\bibitem{YuYa-JOSA13}
{\sc J.~Yuan and J.~Yang}, {\em Computational design for efficient
  second-harmonic generation in nonlinear photonic crystals}, J. Opt. Soc. Am.
  B, 30 (2013), pp.~205--210.

\bibitem{ZeHoLiKoMo-PRB09}
{\sc Y.~Zeng, W.~Hoyer, J.~Liu, S.~W. Koch, and J.~V. Moloney}, {\em Classical
  theory for second-harmonic generation from metallic nanoparticles}, Phys.
  Rev. B, 79 (2009).
\newblock 235109.

\end{thebibliography}

}


\end{document}